\documentclass[12pt,twoside]{amsart}
\usepackage{latexsym,amsfonts,amsmath,amssymb,soul}
\usepackage{enumerate,soul}
\usepackage[usenames, dvipsnames]{color}
\numberwithin{equation}{section}

\makeatletter
\@namedef{subjclassname@2020}{%
  \textup{2020} Mathematics Subject Classification}
\makeatother

\newtheorem{Th}{Theorem}[section]
\newtheorem{Rem}[Th]{Remark}
\newtheorem{Ex}[Th]{Example}
\newtheorem{Lemma}[Th]{Lemma}

\newtheorem{Prop}[Th]{Proposition}
\newtheorem{Cor}[Th]{Corollary}

\catcode`\@=11

\renewcommand{\section}%
   {\setcounter{equation}{0}\@startsection {section}{1}{\z@}{-3.5ex plus -1ex
  minus -.2ex}{2.3ex plus .2ex}{\Large\bf}}

\def\Wig{\mathop{\rm Wig}\nolimits}
\def\ds{\displaystyle}
\def\R{\mathbb R}
\def\C{\mathbb C}
\def\N{\mathbb N}

\newcommand{\D}{\mathcal{D}}

\newcommand{\F}{\mathcal{F}}
\newcommand{\Lin}{\mathcal{L}}
\newcommand{\Sch}{\mathcal{S}}

\newcommand{\beqsn}{\arraycolsep1.5pt\begin{eqnarray*}}
\newcommand{\eeqsn}{\end{eqnarray*}\arraycolsep5pt}
\newcommand{\beqs}{\arraycolsep1.5pt\begin{eqnarray}}
\newcommand{\eeqs}{\end{eqnarray}\arraycolsep5pt}

\title{Mean-dispersion principles and the Wigner transform}

\author[Boiti]{Chiara Boiti}
\address{
Dipartimento di Matematica e Informatica \\Universit\`a di Ferrara\\
Via Ma\-chia\-vel\-li n.~30\\
I-44121 Ferrara\\
Italy}
\email{chiara.boiti@unife.it}

\author[Jornet]{David Jornet}
\address{
Instituto Universitario de Matem\'atica Pura y Aplicada IUMPA\\
Universitat Po\-li\-t\`ecni\-ca de Val\`encia\\
Camino de Vera, s/n\\
E-46071 Valencia\\
Spain}
\email{djornet@mat.upv.es}

\author[Oliaro]{Alessandro Oliaro}
\address{Dipartimento di Matematica\\ Universit\`a di Torino\\
 Via Carlo Alberto n.~10\\ I-10123 Torino\\ Italy}
 \email{alessandro.oliaro@unito.it}

\begin{document}


\begin{abstract}
Given a function $f\in L^2(\R)$, we consider means and variances associated to $f$ and its Fourier transform $\hat{f}$, and explore their relations with the Wigner transform $W(f)$,
obtaining a simple new proof of Shapiro's mean-dispersion principle.
Uncertainty principles for orthonormal sequences in $L^2(\R)$ involving linear partial differential
operators with polynomial coefficients
and the Wigner distribution, or different Cohen class representations, are obtained, and an extension to the case of Riesz bases is studied.
\end{abstract}

\maketitle

\markboth{\sc  Mean-dispersion principles and the Wigner transform}
 {\sc C.~Boiti, D.~Jornet, A.~Oliaro}

\section{Introduction}

This paper treats uncertainty principles for families of orthonormal functions in $L^2(\mathbb{R})$ in connection with time-frequency analysis. When talking about uncertainty principles, in harmonic analysis, one refers to a class of theorems giving limitations on how much a function and its Fourier transform can be both localized at the same time. Different meanings of the word ``localized'' give rise to different uncertainty principles. For instance, referring to the most classical results (see \cite{FS} for a survey), in the Heisenberg uncertainty principle the localization of $f$ and its Fourier transform $\hat{f}$ has to do with their associated variances, in Benedicks \cite{B} it has to do with the measure of their supports, in Donoho-Stark \cite{DS} with the concept of $\varepsilon$-concentration, in Hardy \cite{H} with (exponential) decay at infinity, and so on. There are, moreover, uncertainty principles giving not only limitations on the localization of a single function and its Fourier transform, but on how such limitations behave, becoming stronger and stronger, when adding more and more elements of an orthonormal system in $L^2$. In this paper we focus in particular on results of this type involving means and variances. For $f\in L^2(\R)$ we define the {\em associated mean}
\beqs
\label{muf}
\mu(f):=\frac{1}{\|f\|^2}\int_\R t|f(t)|^2dt
\eeqs
and the {\em associated variance}
\beqs
\label{deltaf}
\Delta^2(f):=\frac{1}{\|f\|^2}\int_\R|t-\mu(f)|^2|f(t)|^2dt;
\eeqs
observe that, for $\| f\|_2=1$, such quantities are the mean and the variance of $|f|^2$. The \emph{dispersion} associated with $f$ is $\Delta(f):=\sqrt{\Delta^2(f)}$. An uncertainty principle for orthonormal sequences, that constitutes the starting point of the present paper, is due to Shapiro. We shall use throughout the paper the notation $\N_0:=\N\cup\{0\}$, and adopt the following normalization of the Fourier transform:
\beqs
\label{Four-intro}
\hat{f}(\xi)=\frac{1}{\sqrt{2\pi}}\int_\R f(t)e^{-it\xi}dt,\qquad \xi\in\R.
\eeqs
\begin{Th}[Shapiro's Mean-Dispersion Principle]
	\label{Shap-orig}
	There does not exist an infinite ortho\-normal sequence $\{ f_k\}_{k\in\N_0}$ in $L^2(\R)$ such that all $\mu(f_k)$, $\mu(\hat{f}_k)$, $\Delta(f_k)$, $\Delta(\hat{f}_k)$ are uniformly bounded.
\end{Th}
This theorem appeared in an unpublished manuscript of Shapiro from 1991; in \cite{P} a stronger result has been proved, namely, there does not exist an orthonormal basis $\{ f_k\}_{k\in\N_0}$ of $L^2(\R)$ such that
\beqsn
\Delta(f_k),\ \Delta(\hat{f}_k),\ \mu(f_k)
\eeqsn
are uniformly bounded, while there exists an orthonormal basis $\{ f_k\}_{k\in\N_0}$ of $L^2(\R)$ such that
\beqsn
\mu(f_k),\ \mu(\hat{f}_k),\ \Delta(f_k)
\eeqsn
are uniformly bounded. Moreover the following quantitative version of Shapiro's Mean-Dispersion Principle is proved in \cite{JP}.
\begin{Th}[{\cite[Theorem 2.3]{JP}}]
	\label{JP23}
	Let $\{ f_k\}_{k\in\N_0}$ be an orthonormal sequence in $L^2(\R)$. Then for every $n\geq 0$
	\beqs
	\label{m-v-intro}
	\sum_{k=0}^n \left(\Delta^2(f_k)+\Delta^2(\hat{f}_k)+|\mu(f_k)|^2+|\mu(\hat{f}_k)|^2\right)\geq (n+1)^2.
	\eeqs
	Equality holds for every $0\leq n\leq n_0$, $n_0\in\N_0$, if and only if there exist $c_k\in\C$ with $|c_k|=1$ such that $f_k=c_k h_k$ for $k=0,\dots,n_0$, where $h_k$ are the Hermite functions on $\R$ defined as follows:
	\beqs
	\label{hermite}
	h_k(t)=\frac{1}{(2^kk!\sqrt{\pi})^{1/2}}e^{-t^2/2}H_k(t),\qquad t\in\R,
	\eeqs
	where $H_k$ is the Hermite polynomial of degree $k$ given by
	\beqsn
	H_k(t)=(-1)^ke^{t^2}\frac{d^k}{dt^k}e^{-t^2},
	\qquad t\in\R.
	\eeqsn
\end{Th}
Observe that \eqref{m-v-intro} differs for a constant from the result in \cite{JP}, due to a different normalization of the Fourier transform. Theorem \ref{Shap-orig} is an easy consequence of Theorem \ref{JP23}; moreover, Theorem \ref{JP23} also says that the limitation on the concentration of $f_k$ and $\hat{f}_k$ become stronger and stronger by adding more and more elements from the orthonormal system, as the lower bound $(n+1)^2$ increases faster than the number of involved functions. \\[0.2cm]
In this paper we study uncertainty principles of mean-dispersion type involving quadratic time-frequency representations applied to the elements of an orthonormal system in $L^2(\R)$. In order to state our main results we need some basic definitions. The classical cross-Wigner distribution is defined as
\beqs
\label{Wfg-intro}
W(f,g)(x,\xi)=\frac{1}{\sqrt{2\pi}}\int_\R f\left(x+\frac t2\right)\overline{g\left(x-\frac t2\right)}
e^{-it\xi}dt,
\qquad f,g\in L^2(\R),
\eeqs
and we set for convenience $W(f):=W(f,f)$. Let moreover $\hat{L}$ be the linear partial differential operator in $\R^2$ defined as
\beqs
\label{Lhat}
\hat L:=\left(\frac12 D_\xi+x\right)^2+\left(\frac12 D_x-\xi\right)^2.
\eeqs
The following result (that we prove in Theorem \ref{th2} and Corollary \ref{cor4} below) constitutes a Mean-Dispersion uncertainty principle associated to the Wigner transform.
\begin{Th}\label{MDWintro}
	Let $\{ f_k\}_{k\in\N_0}$ be an orthonormal sequence in $L^2(\R)$. Then for every $n\geq 0$
	\beqs
	\label{introest}
	\sum_{k=0}^n \langle \hat{L} W(f_k),W(f_k)\rangle \geq (n+1)^2,
	\eeqs
	where as usual $\langle \cdot,\cdot\rangle$ indicates the inner product in $L^2$ (see Section \ref{sec2} for a discussion on the domain of $\hat{L}$ and the corresponding meaning of $\langle \hat{L} W(f_k),W(f_k)\rangle$). Equality in \eqref{introest} holds for every $0\leq n\leq n_0$, $n_0\in\N_0$, if and only if there exist $c_k\in\C$ with $|c_k|=1$ such that $f_k=c_k h_k$, $k=0,\dots,n_0$, where $h_k$ are the Hermite functions \eqref{hermite}.
\end{Th}
We show that Theorem \ref{MDWintro} implies Theorem \ref{JP23} (and then also Theorem \ref{Shap-orig}), and in this sense it can be interpreted as a Mean-Dispersion principle associated to the Wigner transform. The advantage of Theorem \ref{MDWintro} is twofold. First, the proof is simpler than the one of Theorem \ref{JP23} in \cite{JP}. In particular, it does not need the Rayleigh-Ritz technique used there. Moreover, $\hat{L}$ is not the only operator that can be used in \eqref{introest} in order to have Mean-Dispersion principles of the kind of Theorem \ref{MDWintro}. In Sections \ref{sec3} and \ref{sec5} we give more details on this fact. Here, we just point out that we can use instead of $\hat{L}$ the multiplication operator by $x^2+\xi^2$, obtaining that (see Theorem \ref{prop5} below) if $\{ f_k\}_{k\in\N_0}$ is an orthonormal sequence in $L^2(\R)$, then for every $n\geq 0$
\beqs
\label{covintro}
\sum_{k=0}^n \int_{\R^2} (x^2+\xi^2) |W(f_k)(x,\xi)|^2 dx d\xi \geq \frac{(n+1)^2}{2},
\eeqs
and equality is characterized as in Theorem \ref{MDWintro}. We show that if $f_k$ satisfies $\mu(f_k)=\mu(\hat{f}_k)=0$ then the quantity
\beqsn
\int_{\R^2} (x^2+\xi^2) |W(f_k)(x,\xi)|^2 dx d\xi
\eeqsn
is the trace of the covariance matrix of $|W(f_k)(x,\xi)|^2$; then, comparing \eqref{covintro} with \eqref{m-v-intro} (in the case $\mu(f_k)=\mu(\hat{f}_k)=0$) we observe that we have replaced the two variances associated with $f_k$ and $\hat{f}_k$ in \eqref{m-v-intro}, with (a constant times) the trace of the covariance matrix associated with $W(f_k)$, which reflects the fact that $W(f_k)$ includes at the same time both information on $f_k$ and on $\hat{f}_k$. \\[0.1cm]
Other extensions of Theorem \ref{MDWintro} are also studied. Since there are many different time-frequency representations besides the classical Wigner, we consider the so-called {\em Cohen class}, given by all the representations $Q(f,g)$ of the form 
\beqs
\label{Ch-intro}
Q(f,g)=\frac{1}{\sqrt{2\pi}}\,\sigma* W(f,g),\qquad \sigma\in\Sch'(\R^2),\ f,g\in\Sch(\R);
\eeqs
such class contains all the most used time-frequency representations. A natural question is if in Theorem \ref{MDWintro} one can substitute $W(f_k)$ with $Q(f_k):=Q(f_k,f_k)$, and which operators can be considered instead of $\hat{L}$. We prove in Section \ref{sec4} that for a suitable class of {\em kernels} $\sigma$ in \eqref{Ch-intro} a result of the kind of Theorem \ref{MDWintro} can be formulated for representations $Q$ in the Cohen class. Finally, the Mean-Dispersion principle for the Wigner transform can be extended to Riesz bases instead of orthonormal bases. \\[0.2cm]
The paper is organized as follows. In Sections \ref{sec1} and \ref{sec2} we give basic results on the Wigner transform and on the action of the Wigner transform on Hermite functions. In Section \ref{sec3} we prove Theorem \ref{MDWintro}. Section \ref{sec5} is devoted to the study of the case of the covariance matrix associated with $W(f_k)$ and to the proof of \eqref{covintro}. In Sections \ref{sec4} and \ref{sec7} we extend the results to the Cohen class and Riesz bases.

\section{The Wigner distribution}
\label{sec1}

Besides the classical cross-Wigner distribution $W(f,g)$ for $f,g\in L^2(\R)$ defined in \eqref{Wfg-intro}
we also consider the following Wigner-like transform introduced in \cite{BO}
\beqsn
\Wig[u](x,\xi)=\frac{1}{\sqrt{2\pi}}\int_\R u\left(x+\frac t2,x-\frac t2\right)
e^{-it\xi}dt,
\qquad u\in L^2(\R^2),
\eeqsn
with standard extensions to $f,g\in\Sch'(\R)$ and $u\in\Sch'(\R^2)$.
Such operators are strictly related since
\beqsn
W(f,g)=\Wig[f\otimes\bar{g}].
\eeqsn
However, the second one has the advantage, with respect to the classical Wigner transform,
that
\beqsn
\Wig:\ \Sch(\R^2)&&\longrightarrow\Sch(\R^2)\\
\Wig:\ \Sch'(\R^2)&&\longrightarrow\Sch'(\R^2)
\eeqsn
is a linear invertible operator, being composition of a linear invertible change of variables and 
a partial Fourier transform. Indeed, denoting by $\F(f)(\xi)=\hat{f}(\xi)$
the classical Fourier transform \eqref{Four-intro}, by
\beqsn
\F_2(u)(x,\xi)=\frac{1}{\sqrt{2\pi}}\int_\R u(x,t)e^{-it\xi}dt,\qquad (t,\xi)\in\R^2,
\eeqsn
the partial Fourier transform with respect to the second variable, and by
\beqsn
\tau_s u(x,t)=u\left(x+\frac t2,x-\frac t2\right),
\eeqsn
we have that
\beqsn
\Wig[u]=\F_2\tau_s u.
\eeqsn

The inverses of the operators above are
\beqsn
\F^{-1}(F)(x)=\frac{1}{\sqrt{2\pi}}\int_\R F(\xi)e^{ix\xi}d\xi
\eeqsn
and
\beqsn
\tau_s^{-1}F(x,t)=F\left(\frac{x+t}{2},x-t\right).
\eeqsn

Moreover, denoting by
\beqsn
&&M_1u(x,y)=xu(x,y),\quad M_2u(x,y)=yu(x,y),\\
&&D_1u(x,y)=D_xu(x,y),\quad D_2u(x,y)=D_yu(x,y),
\eeqsn
for $D_x=-i\partial_x$ and $D_y=-i\partial_y$, a straightforward computation (see also
\cite{BO}) shows that
\beqs
\label{24W}
&&D_1\Wig[u]=\Wig[(D_1+D_2)u]\\
\label{25W}
&&D_2\Wig[u]=\Wig[(M_2-M_1)u]\\
\label{26W}
&&M_1\Wig[u]=\Wig\left[\frac12(M_1+M_2)u\right]\\
\label{27W}
&&M_2\Wig[u]=\Wig\left[\frac12(D_1-D_2)u\right]
\eeqs
for all $u\in\Sch(\R^2)$. \\
We write $M$ and $D$ for the multiplication and differentiation operators when just one variable is involved, so for $u\in\Sch(\R)$
$$
Mu(t)=tu(t),\quad Du(t)=-iu'(t).
$$

 Moreover we also adopt, for convenience, the following notations. First, we write $\langle\cdot,\cdot\rangle$ to indicate both the inner product in $L^2$, the duality $\Sch'$-$\Sch$ (we consider here distributions as conjugate-linear functionals), and in general the integral
 $$
 \langle g,h\rangle=\int_\mathbb{R} g(t)\overline{h(t)}\,dt
 $$
 each time such integral is finite, even though $g,h$ are not $L^2$ functions. Second, we write
 \beqs
 \label{D1}
 \langle D^{n}f, D^{m}g\rangle
 \eeqs
 for the integral
 \beqs
 \label{D2}
 \int_\mathbb{R} \xi^{n+m}\hat{f}(\xi)\overline{\hat{g}(\xi)}\,d\xi
 \eeqs
 when the last one makes sense and is finite. It coincides with
  $$
 \int_\mathbb{R} D^nf(t)\overline{D^mg(t)}\,dt
 $$
 if $D^nf, D^mg\in L^2$ by {\em Parseval's formula}
 \beqs
\label{Parceval}
\langle f, g\rangle=\langle\hat{f},\hat{g}\rangle,
\qquad\forall f,g\in L^2(\R).
\eeqs
We use the symbol $\langle\cdot,\cdot\rangle$ with analogous meaning in dimension greater than $1$.
 
With this notation, formulas \eqref{24W}-\eqref{27W} hold also for $u\in\Sch'(\R^2)$.
Let us prove, for instance, \eqref{24W}.
Since it's valid in $\Sch(\R^2)$, then for all $u,\varphi\in\Sch(\R^2)$:
\beqs
\nonumber
\langle D_1\Wig[u],\varphi\rangle&&=
\langle\Wig[(D_1+D_2)u],\varphi\rangle
=\langle\F_2\tau_s(D_1+D_2)u,\varphi\rangle\\
\nonumber
&&=\langle\tau_s(D_1+D_2)u,\F_2^{-1}\varphi\rangle
=\langle(D_1+D_2)u,\tau_s^{-1}\F_2^{-1}\varphi\rangle\\
\label{01}
&&=\langle u,(D_1+D_2)(\tau_s^{-1}\F_2^{-1}\varphi)\rangle
\eeqs
by Parseval's formula and
\beqs
\label{tau}
\langle\tau_s u,\tau_s v\rangle=\langle u, v\rangle,
\qquad\forall u,v\in L^2(\R^2).
\eeqs

On the other hand, for all $u,\varphi\in\Sch(\R^2)$,
\beqsn
\langle D_1\Wig[u],\varphi\rangle=\langle\Wig[u],D_1\varphi\rangle
=\langle\F_2\tau_s u,D_1\varphi\rangle
=\langle u, \tau_s^{-1}\F_2^{-1}(D_1\varphi)\rangle,
\eeqsn
which yields, together with \eqref{01},
\beqsn
\tau_s^{-1}\F_2^{-1}(D_1\varphi)=(D_1+D_2)(\tau_s^{-1}\F_2^{-1}\varphi).
\eeqsn
Therefore, if $u\in\Sch'(\R^2)$ and $\varphi\in\Sch(\R^2)$:
\beqsn
\langle D_1\Wig[u],\varphi\rangle&&=
\langle\Wig[u],D_1\varphi\rangle
=\langle\F_2\tau_s u,D_1\varphi\rangle
=\langle u, \tau_s^{-1}\F_2^{-1}(D_1\varphi)\rangle\\
&&=\langle u,(D_1+D_2)(\tau_s^{-1}\F_2^{-1}\varphi)\rangle
=\langle(D_1+D_2)u,\tau_s^{-1}\F_2^{-1}\varphi\rangle\\
&&=\langle\Wig[(D_1+D_2)u], \varphi\rangle,
\eeqsn
so that \eqref{24W} is valid also for $u\in\Sch'(\R^2)$.

Similarly also \eqref{25W}-\eqref{27W} hold for $u\in\Sch'(\R^2)$.

More generally, we have the following result (proved in \cite{BJO-Wigner} for $u\in\Sch(\R^2)$):

\begin{Prop}
\label{prop1}
Let $P(x,y,D_x,D_y)$ be a linear partial differential operator with polynomial coefficients.
Then for all $u\in\Sch'(\R^2)$:
\beqs
\nonumber
&&P(M_1,M_2,D_1,D_2)\Wig[u]=\\
\label{Prop22W}
&&=\Wig\left[P\left(\frac12(M_1+M_2),\frac12(D_1-D_2),D_1+D_2,M_2-M_1\right)u\right],
\eeqs
\beqs
\nonumber
&&\Wig[P(M_1,M_2,D_1,D_2)u]=\\
\label{new319W}
&&=P\left(M_1-\frac12D_2,M_1+\frac12D_2,\frac12D_1+M_2,\frac12D_1-M_2\right)
\Wig[u].
\eeqs
\end{Prop}

 The above proposition will be useful to relate the classical Wigner distribution
 $W(f)$ to the mean \eqref{muf}
 and the variance \eqref{deltaf}
associated with a function $f\in L^2(\R)$ and its Fourier transform $\hat{f}\in L^2(\R)$. 
 
 \begin{Prop}
 \label{prop2}
 Given $f\in L^2(\R)$ with finite associated means and variances of $f$ and $\hat f$, the following properties hold:
 \begin{enumerate}[(a)]
 \item
 $\langle M^2 f,f\rangle=\|f\|^2(\mu^2(f)+\Delta^2(f))$
 \item
 $\langle D^2 f,f\rangle=\|f\|^2(\mu^2(\hat{f})+\Delta^2(\hat{f}))$
 \item
 $\langle M_1 W(f),W(f)\rangle=\|f\|^4\mu(f)$
 \item
 $\langle M_2 W(f),W(f)\rangle=\|f\|^4\mu(\hat f)$
 \item
 $\langle D_1 W(f),W(f)\rangle=0$
  \item
 $\langle D_2 W(f),W(f)\rangle=0$
 \item
 $\langle D^2_1 W(f),W(f)\rangle=2\|f\|^4\Delta^2(\hat f)$
 \item
 $\langle D^2_2 W(f),W(f)\rangle=2\|f\|^4\Delta^2(f)$
 \item
 $\langle M_1D_1 W(f),W(f)\rangle=\frac i2\|f\|^4$\\
 $\langle D_1M_1 W(f),W(f)\rangle=-\frac i2\|f\|^4$
  \item
 $\langle M_2D_2 W(f),W(f)\rangle=\frac i2\|f\|^4$\\
 $\langle D_2M_2 W(f),W(f)\rangle=-\frac i2\|f\|^4$
 \item
 $\langle M_1^2 W(f),W(f)\rangle=\|f\|^4(\mu^2(f)+\frac12\Delta^2(f))$
  \item
 $\langle M_2^2 W(f),W(f)\rangle=\|f\|^4(\mu^2(\hat f)+\frac12\Delta^2(\hat f))$
  \end{enumerate}
 \end{Prop}

\begin{proof}
Let us first recall that \eqref{Parceval} and \eqref{tau} imply the following 
{\em Moyal's formula} for the cross-Wigner distribution (cf. \cite[p.~66]{G})
\beqs
\label{Moyal}
\langle W(f_1,g_1),W(f_2,g_2)\rangle=\langle f_1,f_2\rangle
\overline{\langle g_1,g_2\rangle},\qquad
\forall f_1,f_2,g_1,g_2\in L^2(\R).
\eeqs

Note that the assumption that $f$ has 
finite associated mean and variance implies that $Mf\in L^2(\R)$:
\beqs
\nonumber
\langle Mf,Mf\rangle=&&
\nonumber
\int_\R y^2|f(y)|^2dy
=\int_\R(y-\mu(f)+\mu(f))^2|f(y)|^2dy\\
\nonumber
=&&\int_\R|y-\mu(f)|^2|f(y)|^2dy+2\mu(f)\int_\R(y-\mu(f))|f(y)|^2dy
+\mu^2(f)\|f\|^2\\
\nonumber
=&&\|f\|^2\Delta^2(f)+2\mu^2(f)\|f\|^2-2\mu^2(f)\|f\|^2+\mu^2(f)\|f\|^2\\
\label{M2}
=&&\|f\|^2(\Delta^2(f)+\mu^2(f)).
\eeqs
In the same way, the fact that $\hat{f}$ has finite associated mean and variance implies that $Df\in L^2(\R)$. This means that Moyal's formula \eqref{Moyal} can be applied when, in its left-hand side, $Mf$ or $Df$ appear in the arguments of the Wigner transform.

Now we analyze the case when in the left-hand side of \eqref{Moyal} the expression $W(f,M^2g)$ appears, for $f,g\in L^2(\R)$ with finite associated means and variances of $f, g, \hat{f}, \hat{g}$. Observe that, for  $f,g\in\Sch(\R)$,
\beqsn
W(f,M^2g)(x,\xi)=&&\int_{\mathbb{R}} f\left(x+\frac{t}{2}\right)\overline{\left(x-\frac{t}{2}\right)^2 g\left(x-\frac{t}{2}\right)} e^{-it\xi}\,dt \\
=&&\int_\mathbb{R} \left[2x-\left(x+\frac{t}{2}\right)\right] f\left(x+\frac{t}{2}\right)\overline{\left( x-\frac{t}{2}\right)g\left(x-\frac{t}{2}\right)} e^{-it\xi}\,dt \\
=&& 2xW(f,Mg)(x,\xi)-W(Mf,Mg)(x,\xi).
\eeqsn
Such an equality holds in fact for $f,g\in\Sch'(\R)$ and for tempered distributions it reads
\beqs
\label{WigDistrM}
W(f,M^2g)=2M_1 W(f,Mg)-W(Mf,Mg).
\eeqs
By the observations above, $Mf, Mg\in L^2(\R)$, and so from \eqref{WigDistrM} we have that $W(f,M^2g)$ is a function, and we can consider
\beqsn
\langle W(f,M^2g),W(f,g)\rangle = \int_{\R^2} W(f,M^2g)(x,\xi)\overline{W(f,g)(x,\xi)}\,dx\,d\xi.
\eeqsn
Since $g$ and $Mg$ are $L^2$-functions, we can consider as standard a sequence $g_j\in\Sch(\R)$ such that $g_j\to g$ and $Mg_j\to Mg$ for $j\to\infty$. Since $M^2 g_j\in L^2(\R)$ for every $j\in\N_0$, by \eqref{Moyal} we have 
\beqsn
\langle W(f,M^2g_j),W(f,g)\rangle = \langle f,f\rangle \overline{\langle M^2 g_j,g\rangle} = \langle f,f\rangle \overline{\langle Mg_j,Mg\rangle}
\eeqsn
Then, we have
\beqs
\label{lim1}
\langle W(f,M^2g_j),W(f,g)\rangle \to \langle f,f\rangle \overline{ \langle Mg,Mg\rangle} = \langle f,f\rangle \overline{ \langle M^2g,g\rangle}
\eeqs
as $j\to\infty$. On the other hand, by \eqref{WigDistrM} and \eqref{26W} we get
\beqsn
\langle W(f,M^2g_j),W(f,g)\rangle =&& \langle 2M_1 W(f,Mg_j)-W(Mf,Mg_j),W(f,g)\rangle \\
=&&\langle W(f,Mg_j),2M_1 W(f,g)\rangle-\langle W(Mf,Mg_j),W(f,g)\rangle \\
=&&\langle W(f,Mg_j),W(Mf,g)+W(f,Mg)\rangle-\langle W(Mf,Mg_j),W(f,g)\rangle.
\eeqsn
Since $g_j\to g$, $Mg_j\to Mg$, and $f,g,Mf,Mg\in L^2(\R)$, by the $L^2$-continuity of the Wigner transform we have
\beqsn
\langle W(f,M^2g_j),W(f,g)\rangle\to \langle W(f,Mg),W(Mf,g)+W(f,Mg)\rangle-\langle W(Mf,Mg),W(f,g)\rangle
\eeqsn
as $j\to\infty$; by the same calculations as above we get
\beqs
\label{lim2}
\langle W(f,M^2g_j),W(f,g)\rangle\to \langle W(f,M^2g),W(f,g)\rangle
\eeqs
as $j\to\infty$. From \eqref{lim1} and \eqref{lim2} we then have that $\langle W(f,M^2g),W(f,g)\rangle$ is a convergent integral and
\beqs
\label{concl1}
\langle W(f,M^2g),W(f,g)\rangle = \langle f,f\rangle \overline{ \langle M^2g,g\rangle}.
\eeqs

Recall now that for every $u,v\in\Sch'(\R)$ the following formula holds
\beqs
\label{WigFourier}
W(\hat{u},\hat{v})(x,\xi)=W(u,v)(-\xi,x);
\eeqs
then, since $\hat{f}$ and $\hat{g}$ have finite associated means and variances, the same procedure can be applied when we have $W(f,D^2g)$ instead of $W(f,M^2g)$ obtaining that, with the notation \eqref{D1}-\eqref{D2},
\beqs
\label{concl2}
\langle W(f,D^2g),W(f,g)\rangle = \langle W(\hat{f},M^2\hat{g}),W(\hat{f},\hat{g})\rangle= \langle f,f\rangle\overline{\langle D^2g,g\rangle}.
\eeqs

Similar considerations can be done for $MDf$, since
\beqsn
W(MDf,g)=&&\int\left(x+\frac t2\right)Df\left(x+\frac t2\right)
\overline{g\left(x-\frac t2\right)}e^{-it\xi}dt\\
=&&\int\left[2x-\left(x-\frac t2\right)\right]Df\left(x+\frac t2\right)
\overline{g\left(x-\frac t2\right)}e^{-it\xi}dt\\
=&&2xW(Df,g)-W(Df,Mg)
\eeqsn
is a function, being $Df,Mg\in L^2(\R)$ under the assumptions of finite associated means and variances.
Arguing as for $M^2f$ we then have
\beqs
\label{C2}
\langle W(MDf,g),W(f,g)\rangle
=\langle MDf,f\rangle\overline{\langle g,g\rangle}
=\langle Df,Mf\rangle\overline{\langle g,g\rangle}.
\eeqs

All the above considerations will be implicit from now on.

Let us now prove point $(a)$:
it follows from \eqref{M2} since $\langle M^2f,f\rangle=\langle Mf,Mf\rangle$.

$(b)$:
With the notations \eqref{D1}-\eqref{D2}, by point $(a)$ applied to $\hat f$:

\beqsn
\langle D^2f,f\rangle=\langle\xi^2\hat f,\hat f\rangle=\|\hat f\|^2(\mu^2(\hat f)+\Delta^2(\hat f))
=\| f\|^2(\mu^2(\hat f)+\Delta^2(\hat f))
\eeqsn

$(c)$:
From \eqref{26W} and Moyal's formula \eqref{Moyal}:

\beqsn
\langle M_1 W(f),W(f)\rangle=&&\langle M_1\Wig[f\otimes\bar{f}],W(f)\rangle\\
=&&\langle\Wig[\frac12(M_1+M_2)(f\otimes\bar{f})],W(f)\rangle\\
=&&\frac12(\langle W(Mf,f),W(f,f)\rangle+\langle W(f,Mf),W(f,f)\rangle)\\
=&&\frac12(\langle Mf,f\rangle\overline{\langle f,f\rangle}
+\langle f,f\rangle\overline{\langle Mf,f\rangle})
=\|f\|^4\mu(f),
\eeqsn
since $\mu(f)\in\R$.

$(d)$:
From \eqref{27W}, Moyal's and Parseval's formulas
\eqref{Moyal} and \eqref{Parceval}:

\beqsn
\langle M_2 W(f),W(f)\rangle
=&&\langle\Wig[\frac12(D_1-D_2)f\otimes\bar f],W(f)\rangle\\
=&&\frac12(\langle W(Df,f),W(f,f)\rangle+\langle W(f,Df),W(f,f)\rangle)\\
=&&\frac12(\langle Df,f\rangle\overline{\langle f, f\rangle}+\langle f,f\rangle
\overline{\langle Df,f\rangle})\\
=&&\frac12(\langle\xi\hat f,\hat f\rangle\|f\|^2+\|f\|^2\overline{\langle\xi\hat f,\hat f\rangle})
=\|f\|^4\mu(\hat f),
\eeqsn

since $\mu(\hat f)\in\R$.

$(e)$:
From \eqref{24W}, \eqref{Moyal} and \eqref{Parceval}:

\beqsn
\langle D_1 W(f),W(f)\rangle
=&&\langle\Wig[(D_1+D_2)f\otimes\bar{f}],W(f)\rangle\\
=&&\langle W(Df,f)-W(f,Df),W(f,f)\rangle\\
=&&\langle Df,f\rangle\overline{\langle f,f\rangle}-\langle f,f\rangle\overline{\langle Df,f\rangle}\\
=&&\langle\xi\hat f, \hat f\rangle\|f\|^2-\|f\|^2\overline{\langle\xi\hat f,\hat f\rangle}
=0.
\eeqsn

$(f)$:
From \eqref{25W} and \eqref{Moyal}:

\beqsn
\langle D_2W(f),W(f)\rangle
=&&\langle\Wig[(M_2-M_1)f\otimes\bar f],W(f)\rangle\\
=&&\langle W(f,Mf)-W(Mf,f),W(f,f)\rangle\\
=&&\langle f,f\rangle\overline{\langle Mf,f\rangle}-\langle Mf,f\rangle\overline{\langle f,f\rangle}
=0.
\eeqsn

$(g)$:
From \eqref{24W}, \eqref{Moyal}, \eqref{concl2}, \eqref{Parceval} and point $(a)$:

\beqsn
\langle D_1^2W(f),W(f)\rangle
=&&\langle\Wig[(D_1+D_2)^2f\otimes\bar f],W(f)\rangle\\
=&&\langle W(D^2f,f)-2W(Df,Df)+W(f,D^2f),W(f,f)\rangle\\
=&&\langle D^2f,f\rangle\overline{\langle f,f\rangle}
-2\langle Df,f\rangle\overline{\langle Df,f\rangle}+
\langle f,f\rangle\overline{\langle D^2f,f\rangle}\\
=&&\langle\xi^2\hat f, \hat f\rangle\|f\|^2-2|\langle\xi\hat f, \hat f\rangle|^2+\|f\|^2
\overline{\langle\xi^2\hat f,\hat f\rangle}\\
=&&2\|f\|^2\|\hat f\|^2(\mu^2(\hat f)+\Delta^2(\hat f))-2\mu^2(\hat f)\|\hat f\|^4
=2\|f\|^4\Delta^2(\hat f).
\eeqsn

$(h)$:
From \eqref{25W}, \eqref{Moyal}, \eqref{concl1} and point $(a)$:
\beqsn
\langle D_2^2 W(f),W(f)\rangle
=&&\langle \Wig[(M_2-M_1)^2f\otimes\bar f],W(f)\rangle\\
=&&\langle W(f,M^2f)-2W(Mf,Mf)+W(M^2f,f),W(f,f)\rangle\\
=&&\langle f,f\rangle\overline{\langle M^2f,f\rangle}-2\langle Mf,f\rangle
\overline{\langle Mf,f\rangle}+\langle M^2f,f\rangle\overline{\langle f,f\rangle}\\
=&&2\|f\|^4(\mu^2(f)+\Delta^2(f))-2\|f\|^4\mu^2(f)
=2\|f\|^4\Delta^2(f).
\eeqsn

$(i)$:
From \eqref{24W}, \eqref{26W}, \eqref{Moyal}, \eqref{C2} and \eqref{Parceval}:
\beqsn
&&\hskip-1cm \langle M_1D_1W(f),W(f)\rangle
=\langle\Wig[\frac12(M_2+M_1)(D_1+D_2)f\otimes\bar f],W(f)\rangle\\
=&&\frac12\langle\Wig[(M_2D_1+M_1D_1+M_2D_2+M_1D_2)f\otimes\bar f],W(f)\rangle\\
=&&\frac12\langle W(Df,Mf)+W(MDf,f)-W(f,MDf)-W(Mf,Df),W(f,f)\rangle\\
=&&\frac12(\langle Df,f\rangle\overline{\langle Mf,f\rangle}
+\langle Df,Mf\rangle\overline{\langle f,f\rangle}-\langle f,f\rangle
\overline{\langle Df,Mf\rangle}
-\langle Mf,f\rangle\overline{\langle Df,f\rangle})\\
=&&\frac12(\langle\xi\hat f,\hat f\rangle\overline{\mu(f)}\|f\|^2
+\|f\|^2(\langle Df,Mf\rangle-\overline{\langle Df,Mf\rangle})
-\|f\|^2\mu(f)\overline{\langle\xi\hat f,\hat f\rangle}).
\eeqsn
Since $\mu(f)\in\R$, $\langle\xi\hat f,\hat f\rangle=\mu(\hat f)\|f\|^2\in\R$ and
\beqs
\nonumber
\langle Df,Mf\rangle=&&\langle f,DMf\rangle
=i\langle f,f\rangle+\langle f,MDf\rangle\\
\label{DfMf}
=&&i\|f\|^2+\langle Mf,Df\rangle
=i\|f\|^2+\overline{\langle Df,Mf\rangle}
\eeqs
we finally have that
\beqsn
\langle M_1D_1 W(f),W(f)\rangle=\frac i2\|f\|^4.
\eeqsn
Therefore
\beqsn
\langle D_1M_1W(f),W(f)\rangle=&&
\langle M_1W(f),D_1W(f)\rangle
=\langle W(f),M_1D_1 W(f)\rangle\\
=&&\overline{\langle M_1D_1 W(f),W(f)\rangle}
=-\frac i2\|f\|^4.
\eeqsn

$(j)$:
From \eqref{25W}, \eqref{27W}, \eqref{Moyal}, \eqref{C2}, \eqref{Parceval} and \eqref{DfMf}:
\beqsn
&&\hskip-1cm\langle M_2D_2 W(f),W(f)\rangle
=\langle\Wig[\frac12(D_1-D_2)(M_2-M_1)f\otimes\bar f],W(f)\rangle\\
=&&\frac12
\langle\Wig[(D_1M_2-D_2M_2-D_1M_1+D_2M_1)f\otimes\bar f],W(f)\rangle\\
=&&\frac12\langle W(Df,Mf)-\frac1iW(f,f)+W(f,MDf),W(f,f)\rangle\\
&&-\frac12\langle\frac1iW(f,f)
+W(MDf,f)
+W(Mf,Df),W(f,f)\rangle\\
=&&\frac12(\langle Df,f\rangle\overline{\langle Mf,f\rangle}
-\frac 1i\langle f,f\rangle\overline{\langle f,f\rangle}+\langle f,f\rangle
\overline{\langle Df,Mf\rangle})\\
&&-\frac12(\frac1i\|f\|^4+
\langle Df,Mf\rangle\overline{\langle f,f\rangle}
+\langle Mf,f\rangle\overline{\langle Df,f\rangle})\\
=&&\frac12\langle\xi\hat f,\hat f\rangle\mu(f)\|f\|^2+i\|f\|^4+\frac12\|f\|^2
(\overline{\langle Df,Mf\rangle}-\langle Df,Mf\rangle)\\
&&-\frac12\|f\|^2\mu(f)\overline{\langle\xi\hat f,\hat f\rangle}
=i\|f\|^4-\frac i2\|f\|^4=\frac i2\|f\|^4.
\eeqsn
It follows that
\beqsn
\langle D_2M_2W(f),W(f)\rangle=&&
\frac1i\langle W(f,f),W(f,f)\rangle+\langle M_2D_2W(f),W(f)\rangle\\
=&&-i\|f\|^4+\frac i2\|f\|^4=-\frac i2\|f\|^4.
\eeqsn

$(k)$:
From \eqref{26W}, \eqref{Moyal} and point (a):
\beqsn
&&\hskip-1cm\langle M_1^2W(f),W(f)\rangle
=\langle M_1\Wig[f\otimes\bar f],M_1\Wig[f\otimes\bar f]\rangle\\
=&&\langle\Wig[\frac12(M_2+M_1)f\otimes\bar f],\Wig[\frac12(M_2+M_1)f\otimes\bar f]\rangle\\
=&&\frac14\langle
W(f,Mf)+W(Mf,f),W(f,Mf)+W(Mf,f)\rangle\\
=&&\frac14(\langle f,f\rangle\overline{\langle Mf,Mf\rangle}
+\langle Mf,f\rangle\overline{\langle f,Mf\rangle}
+\langle f,Mf\rangle\overline{\langle Mf,f\rangle}+
\langle Mf,Mf\rangle\overline{\langle f,f\rangle})\\
=&&\frac14(\|f\|^2\overline{\langle M^2f,f\rangle}
+2\mu^2(f)\|f\|^4+\langle M^2 f,f\rangle\|f\|^2)\\
=&&\|f\|^4\left(\frac12\Delta^2(f)+\mu^2(f)\right).
\eeqsn

$(l)$:
From \eqref{27W}, \eqref{Moyal}, \eqref{Parceval} and point (b):

\beqsn
&&\hskip-1cm\langle M_2^2W(f),W(f)\rangle=\langle M_2\Wig[f\otimes\bar f],
M_2\Wig[f\otimes\bar f]\rangle\\
=&&\langle\Wig[\frac12(D_1-D_2)f\otimes\bar f],\Wig[\frac12(D_1-D_2)f\otimes\bar f]\rangle\\
=&&\frac14\langle W(Df,f)+W(f,Df),W(Df,f)+W(f,Df)\rangle\\
=&&\frac14(\langle Df,Df\rangle\overline{\langle f,f\rangle}
+\langle f,Df\rangle\overline{\langle Df,f\rangle}
+\langle Df,f\rangle\overline{\langle f,Df\rangle}+\langle f,f\rangle
\overline{\langle Df,Df\rangle})\\
=&&\frac14(\langle D^2 f,f\rangle\|f\|^2
+\langle\hat f,\xi\hat f\rangle\overline{\langle\xi\hat f,\hat f\rangle}
+\langle\xi\hat f,\hat f\rangle\overline{\langle\hat f, \xi\hat f\rangle}
+\|f\|^2\overline{\langle D^2f,f\rangle})\\
=&&\|f\|^4\left(\frac12\Delta^2(\hat f)+\mu^2(\hat f)\right).
\eeqsn
The proof is complete.
\end{proof}

\begin{Cor}
\label{cor1}
Given $f\in L^2(\R)$ with $\|f\|=1$ and finite associated mean and variance of $f$ and $\hat f$, the following properties hold:
\begin{enumerate}[(a)]
 \item
 $\langle M^2 f,f\rangle=\mu^2(f)+\Delta^2(f)$
 \item
 $\langle D^2 f,f\rangle=\mu^2(\hat{f})+\Delta^2(\hat{f})$
 \item
 $\langle M_1 W(f),W(f)\rangle=\mu(f)$
 \item
 $\langle M_2 W(f),W(f)\rangle=\mu(\hat f)$
 \item
 $\langle D_1 W(f),W(f)\rangle=0$
  \item
 $\langle D_2 W(f),W(f)\rangle=0$
 \item
 $\langle D^2_1 W(f),W(f)\rangle=2\Delta^2(\hat f)$
 \item
 $\langle D^2_2 W(f),W(f)\rangle=2\Delta^2(f)$
 \item
 $\langle M_1D_1 W(f),W(f)\rangle=\frac i2$\\
 $\langle D_1M_1 W(f),W(f)\rangle=-\frac i2$
  \item
 $\langle M_2D_2 W(f),W(f)\rangle=\frac i2$\\
 $\langle D_2M_2 W(f),W(f)\rangle=-\frac i2$
 \item
 $\langle M_1^2 W(f),W(f)\rangle=\mu^2(f)+\frac12\Delta^2(f)$
  \item
 $\langle M_2^2 W(f),W(f)\rangle=\mu^2(\hat f)+\frac12\Delta^2(\hat f)$.
  \end{enumerate}
\end{Cor}

\section{The Hermite basis}
\label{sec2}

For $k\in\N_0=\N\cup\{0\}$, let $h_k$ be the Hermite functions on $\R$ defined by \eqref{hermite}.
It is well known that $h_k$ are eigenfunctions of the Fourier transform and form an orthonormal basis in $L^2(\R)$.
Moreover they are an absolute basis in $\Sch(\R)$ (see \cite{L}).

Denoting by 
\beqsn
h_{j,k}:=\F^{-1}W(h_j,h_k),
\eeqsn
by \cite[Thms. 3.2 and 3.4]{Wong} we have that the functions $\{h_{j,k}\}_{j,k\in\N_0}$
form an orthonormal basis in $L^2(\R^2)$ and are eigenfunctions of the twisted Laplacian:
\beqsn
Lh_{j,k}(y,t)=(2k+1)h_{j,k}(y,t),\qquad j,k\in\N_0,
\eeqsn
for
\beqsn
L:=\left(D_y-\frac12 t\right)^2+\left(D_t+\frac12 y\right)^2.
\eeqsn

By Fourier transform (see \cite[Ex. 3.20]{BJO-realPW})
\beqs
\label{hhat}
\hat h_{j,k}(x,\xi)=W(h_j,h_k)(x,\xi)
\eeqs
are eigenfunctions of the operator $\hat{L}$ defined in \eqref{Lhat},
with the same eigenvalues as before, in the sense that
\beqs
\label{autovettLhat}
\hat L\hat h_{j,k}=(2k+1)\hat h_{j,k}.
\eeqs
Note that also $\{\hat h_{j,k}\}_{j,k\in\N_0}$ are in $\Sch(\R)$ and form an orthonormal basis in $L^2(\R^2)$.

More in general, following the same ideas as in
\cite[Thm. 21.2]{Wong-Weyl}, we can prove:
\begin{Th}
\label{th1}
If $\{f_k\}_{k\in\N_0}$ is an orthonormal basis in $L^2(\R)$, then $\{W(f_j,f_k)\}_{j,k\in\N_0}$
is an orthonormal basis in $L^2(\R^2)$.
\end{Th}

\begin{proof}
Let us first remark that if $\{f_k\}_k$ is an othonormal sequence in $L^2(\R)$ then
$\{W(f_j,f_k)\}_{j,k}$ is an orthonormal sequence in $L^2(\R^2)$ since, by
\eqref{Moyal},
\beqsn
\langle W(f_j,f_k),W(f_i,f_h)\rangle=&&\langle f_j,f_i\rangle
\overline{\langle f_k,f_h\rangle}\\
=&&\delta_{j,i}\cdot\delta_{k,h}
=\begin{cases}
1& \mbox{if} \ (j,k)=(i,h)\cr
0& \mbox{if}\  (j,k)\neq(i,h).
\end{cases}
\eeqsn

In order to prove that $\{W(f_j,f_k)\}_{j,k\in\N_0}$ is a basis for $L^2(\R^2)$, by 
\cite[Thm. 3.4.2]{C}, it is enough to prove that if $F\in L^2(\R^2)$ is such that
\beqs
\label{int0}
\int_{\R^2}F(x,\xi)W(f_j,f_k)(x,\xi)dxd\xi=0,
\qquad\forall j,k\in\N_0,
\eeqs
then $F=0$ a.e. in $\R^2$.

By \cite[Thms. 4.4 and 7.5]{Wong-Weyl} the operator
\beqsn
L^2(\R^2)\ &&\longrightarrow\Lin(L^2(\R),L^2(\R))\\
F\ &&\longmapsto W_F
\eeqsn
defined by
\beqsn
\langle W_F\varphi,\psi\rangle=\frac{1}{\sqrt{2\pi}}\int_{\R^2}F(x,\xi)W(\varphi,\psi)(x,\xi)dxd\xi
\eeqsn
is a bounded linear operator satisfying
\beqs
\label{th75W}
\|W_F\|_{\Lin(L^2,L^2)}\leq\frac{1}{\sqrt{2\pi}}\|F\|_{L^2(\R^2)}=\|W_F\|_{HS},
\eeqs
where $\|\cdot\|_{HS}$ is the Hilbert-Schmidt norm defined by (see
\cite[formula (7.1)]{Wong-Weyl}):
\beqs
\label{71W}
\|W_F\|^2_{HS}:=\sum_{j=0}^{+\infty}\|W_F f_j\|_{L^2(\R)}^2
\eeqs
for an orthonormal basis $\{f_j\}_{j\in\N_0}$ of $L^2(\R)$. The operator $W_F$ is in fact the classical Weyl operator with symbol $F$. Then
\beqsn
\langle W_F f_j,f_k\rangle=\frac{1}{\sqrt{2\pi}}\int_{\R^2}F(x,\xi)W(f_j,f_k)(x,\xi)dxd\xi=0,
\qquad\forall j,k\in\N_0,
\eeqsn
by assumption, which implies that
\beqs
\label{WFfj}
W_Ff_j=0,\qquad\forall j\in\N_0,
\eeqs
since $\{f_j\}_{j\in\N_0}$ is an orthonormal basis in $L^2(\R)$.

From \eqref{th75W} and \eqref{71W} we finally have that
$F=0$ a.e. in $\R^2$.
\end{proof}

The operator $\hat{L}$ defined in \eqref{Lhat} is unbounded on $L^2(\R^2)$ 
(see Remark~\ref{rem66} below) and defined (at least) in $\Sch(\R^2)\subset L^2(\R)$. Now, since the functions \eqref{hhat} are an orthonormal basis for $L^2(\R^2)$, every element $F\in L^2(\R^2)$ can be written as
\beqsn
F=\sum_{j,k=0}^{+\infty} c_{j,k} \hat{h}_{j,k}
\eeqsn
where $c_{j,k}=\langle F,\hat{h}_{j,k}\rangle$. Then, writing
\beqsn
F_N=\sum_{j,k=0}^{N} c_{j,k} \hat{h}_{j,k}\in\Sch(\R^2)
\eeqsn
we have from \eqref{autovettLhat}
\beqsn
\hat{L}F_N(x,\xi)
 = \sum_{j,k=0}^{N} c_{j,k} (2k+1)\hat{h}_{j,k}(x,\xi).
\eeqsn
The operator $\hat{L}$ is then the unbounded and densely defined operator with domain
\beqsn
D(\hat{L})=\{F\in L^2(\R^2) : \sum_{j,k=0}^{+\infty} c_{j,k} (2k+1)\hat{h}_{j,k}\ \text{converges\ in}\ L^2(\R^2)\}
\eeqsn
for $c_{j,k}=\langle F,\hat{h}_{j,k}\rangle$, acting on $F\in D(\hat{L})$ as
\beqsn
\hat{L}F=\sum_{j,k=0}^{+\infty} c_{j,k} (2k+1)\hat{h}_{j,k}\in L^2(\R^2).
\eeqsn
In this case
\beqsn
\langle\hat L F,F\rangle=&&\lim_{N\to+\infty}
\sum_{j,k,j',k'=0}^N\langle c_{j,k}(2k+1)\hat{h}_{j,k},c_{j',k'}\hat{h}_{j',k'}\rangle\\
=&&\lim_{N\to+\infty}
\sum_{j,k=0}^N|c_{j,k}|^2(2k+1)
=\sum_{j,k=0}^{+\infty}|c_{j,k}|^2(2k+1).
\eeqsn
In general we shall write
\beqs
\label{fuoriDL}
\langle\hat L F,F\rangle=\sum_{j,k=0}^{+\infty}|c_{j,k}|^2(2k+1),
\qquad\forall F\in L^2(\R),
\eeqs
meaning that $\langle\hat L F,F\rangle=+\infty$ if the series diverges.
Note that, being $\{\hat h_{j,k}\}_{j,k\in\N_0}$ an orthonormal basis for $L^2(\R^2)$, we have that
$F\in D(\hat L)$ if and only if $\{c_{j,k}(2k+1)\}_{j,k\in\N_0}\in\ell^2$.
This implies that the series \eqref{fuoriDL} converges (but not vice versa).

\section{Mean-Dispersion Principle}
\label{sec3}

From the results of the previous sections we obtain now an alternative formulation and a simple proof of the Shapiro's Mean-Dispersion Principle (see \cite{JP} and the references therein).
To this aim let us first prove some preliminary results.

\begin{Lemma}
\label{lemma1}
Let $\{h_k\}_{k\in\N_0}$ be the Hermite functions defined in \eqref{hermite} and
$\hat L$ as in \eqref{Lhat}. Then for every $j\in\N_0$ we have
\beqsn
\sum_{k=0}^n\langle\hat L W(h_j,h_k),W(h_j,h_k)\rangle=(n+1)^2, \qquad\forall n\in\N_0.
\eeqsn
\end{Lemma}

\begin{proof}
From \eqref{autovettLhat} for all $j,k\in\N_0$ we have
\beqsn
\langle\hat L W(h_j,h_k),W(h_j,h_k)\rangle=
\langle\hat L\hat h_{j,k},\hat h_{j,k}\rangle
=\langle(2k+1)\hat h_{j,k},\hat h_{j,k}\rangle=2k+1,
\eeqsn
since $\{\hat h_{j,k}\}_{j,k}$ is an orthonormal basis in $L^2(\R^2)$.

It follows that
\beqsn
\sum_{k=0}^n\langle\hat LW(h_j,h_k),W(h_j,h_k)\rangle
=\sum_{k=0}^n(2k+1)=(n+1)^2,
\eeqsn
where the last equality is the formula for the sum of all odd numbers from 1 to $2n+1$.
\end{proof}

\begin{Lemma}
\label{lemma2}
Let $\hat L$ be the operator in \eqref{Lhat}. Then for all $f,g\in L^2(\R)$
with finite associated mean and variances of $f,g,\hat{f},\hat g$:
\begin{enumerate}[(i)]
\item
$\ds \hat L W(f,g)=W(f,(M^2+D^2)g)$,
\item
$\ds \langle\hat LW(f,g),W(f,g)\rangle
=\|f\|^2\|g\|^2(\Delta^2(g)+\Delta^2(\hat g)+\mu^2(g)+\mu^2(\hat g))$.
\end{enumerate}
In particular, $\langle\hat LW(f,g),W(f,g)\rangle\in\R$ and
if $\|f\|=\|g\|=1$ then
\beqsn
\langle\hat LW(f,g),W(f,g)\rangle
=\Delta^2(g)+\Delta^2(\hat g)+\mu^2(g)+\mu^2(\hat g).
\eeqsn
\end{Lemma}

\begin{proof}
$(i)$:\ From \eqref{Prop22W} we have 
\beqsn
&&\hskip-1cm\hat LW(f,g)=\left[\left(\frac12D_2+M_1\right)^2+\left(\frac12D_1-M_2\right)^2\right]\Wig[f\otimes\bar g]\\
=&&\Wig\left[\left(\left(\frac12(M_2-M_1)+\frac12(M_2+M_1)\right)^2+\left(\frac12
(D_1+D_2)-\frac12(D_1-D_2)\right)^2\right)f\otimes\bar g\right]\\
=&&\Wig[(M_2^2+D_2^2)f\otimes\bar g]=W(f,(M^2+D^2)g).
\eeqsn

$(ii)$: From $(i)$, \eqref{concl1}, \eqref{concl2}, and Proposition~\ref{prop2}$(a),(b)$:
\beqsn
\langle\hat LW(f,g),W(f,g)\rangle=&&
\langle W(f,(M^2+D^2)g),W(f,g)\rangle\\
=&&\langle W(f,M^2g),W(f,g)\rangle+\langle W(f,D^2g),W(f,g)\rangle\\
=&&\langle f,f\rangle\overline{\langle M^2g,g\rangle}
+\langle f,f\rangle\overline{\langle D^2g,g\rangle}\\
=&&\|f\|^2\|g\|^2(\Delta^2(g)+\mu^2(g))+\|f\|^2\|g\|^2(\Delta^2(\hat g)+\mu^2(\hat g))\\
=&&\|f\|^2\|g\|^2(\Delta^2(g)+\Delta^2(\hat g)+\mu^2(g)+\mu^2(\hat g)).
\eeqsn
\end{proof}

\begin{Th}
\label{th2}
Let $\{f_k\}_{k\in\N_0}$ be such that $\| f_k\|=1$ for every $k\in\N_0$, and let $\{g_k\}_{k\in\N_0}$ be an orthonormal sequence in $L^2(\R)$.
Then
\beqs
\label{T20}
\sum_{k=0}^n\langle\hat LW(f_i,g_k),W(f_i,g_k)\rangle\geq(n+1)^2,
\qquad\forall i,n\in\N_0.
\eeqs
\end{Th}

\begin{proof}
Since $W(f_i,g_k)\in L^2(\R^2)$ and the sequence $\{\hat h_{j,\ell}\}=\{W(h_j,h_\ell)\}$
defined in \eqref{hhat} is an orthonormal basis in $L^2(\R^2)$, we can write
\beqsn
W(f_i,g_k)=\sum_{j,\ell=0}^{+\infty}c^{(i,k)}_{j,\ell}W(h_j,h_\ell)
\eeqsn
with
\beqs
\label{T21}
c^{(i,k)}_{j,\ell}=&&\langle W(f_i,g_k),W(h_j,h_\ell)\rangle
=\langle f_i,h_j\rangle\overline{\langle g_k,h_\ell\rangle},
\eeqs
by \eqref{Moyal}. As in \eqref{fuoriDL} we have 
\beqs
\label{T22}
\sum_{k=0}^n\langle\hat LW(f_i,g_k),W(f_i,g_k)\rangle
=
\sum_{k=0}^n\sum_{j,\ell=0}^{+\infty}|c^{(i,k)}_{j,\ell}|^2
(2\ell+1),
\eeqs
and we can assume that for every $0\leq k\leq n$ the series in \eqref{T22} converges,
otherwise \eqref{T20} would be trivial, being the left-hand side equal to $+\infty$.

By \eqref{T21} and \eqref{T22}, we get
\beqs
\nonumber
&&\hskip-1cm\sum_{k=0}^n\langle\hat LW(f_i,g_k),W(f_i,g_k)\rangle
=\sum_{k=0}^n\sum_{j,\ell=0}^{+\infty}|\langle f_i,h_j\rangle|^2|\langle g_k,h_\ell\rangle|^2
(2\ell+1)\\
\label{T23}
=&&\sum_{j=0}^{+\infty}|\langle f_i,h_j\rangle|^2\sum_{\ell=0}^{+\infty}
\sum_{k=0}^n|\langle g_k,h_\ell\rangle|^2(2\ell+1)
=\sum_{\ell=0}^{+\infty}\left(\sum_{k=0}^n|\langle g_k,h_\ell\rangle|^2\right)(2\ell+1),
\eeqs
since $\|f_i\|^2=1$. Setting
\beqsn
\alpha_\ell:=\sum_{k=0}^n|\langle g_k,h_\ell\rangle|^2,
\eeqsn
we remark that
\beqs
\label{T24}
\sum_{\ell=0}^{+\infty}\alpha_\ell=\sum_{\ell=0}^{+\infty}
\sum_{k=0}^n|\langle g_k,h_\ell\rangle|^2
=\sum_{k=0}^n\sum_{\ell=0}^{+\infty} |\langle g_k,h_\ell\rangle|^2
=\sum_{k=0}^n\|g_k\|^2=n+1.
\eeqs

But for each $\ell\in\N_0$ 
\beqsn
\alpha_\ell\leq\sum_{k=0}^{+\infty}|\langle g_k,h_\ell\rangle|^2
\leq\|h_\ell\|^2=1,
\eeqsn
so that from \eqref{T24} we can write
\beqsn
n+1=\sum_{\ell=0}^{+\infty}\alpha_\ell=\alpha_0+\ldots+\alpha_n
+R_n
\eeqsn
for a reminder
\beqs
\label{reminder}
R_n=\sum_{\ell=n+1}^{+\infty}\alpha_\ell.
\eeqs
Note that
$\alpha_0=\ldots\alpha_n=1$ if $R_n=0$.

For all $0\leq k\leq n$ we set
\beqs
\label{ck}
c_k=
\begin{cases}
 0, & \mbox{if}\ R_n=0\cr
 \frac{1-\alpha_k}{R_n}, &\mbox{if} \ R_n>0.
\end{cases}
\eeqs
Then
\beqs
\label{DD1}
\alpha_k+c_kR_n=1\qquad\forall 0\leq k\leq n
\eeqs
and $(c_0+\ldots+c_n)R_n=R_n$, so that
\beqsn
c_0+\ldots+c_n=
\begin{cases}
1 & \mbox{if}\ R_n>0\cr
0& \mbox{if}\ R_n=0
\end{cases}
\eeqsn
and we can write
\beqs
\label{DD2}
(c_0+\ldots+c_n)\sum_{\ell=n+1}^{+\infty}\alpha_\ell(2\ell+1)=
\sum_{\ell=n+1}^{+\infty}\alpha_\ell(2\ell+1),
\eeqs
being $R_n=0$ iff $\alpha_\ell=0$ for all $\ell\geq n+1$.

We use \eqref{DD2} and \eqref{DD1} in \eqref{T23} to get
\beqs
\nonumber
&&\hskip-1cm\sum_{k=0}^n\langle\hat LW(f_i,g_k),W(f_i,g_k)\rangle
=\sum_{\ell=0}^{+\infty}\alpha_\ell(2\ell+1)\\
\nonumber
=&&\sum_{\ell=0}^n\alpha_\ell(2\ell+1)+(c_{0}+\ldots+c_{n})\sum_{\ell=n+1}^{+\infty}
\alpha_\ell(2\ell+1)\\
\nonumber
=&&\sum_{\ell=0}^n\alpha_\ell(2\ell+1)+c_0\sum_{\ell=n+1}^{+\infty}
\alpha_\ell\underbrace{(2\ell+1)}_{\geq1}
+c_1\sum_{\ell=n+1}^{+\infty}
\alpha_\ell\underbrace{(2\ell+1)}_{\geq3}\\
\nonumber
&&\ldots+c_{n-1}\sum_{\ell=n+1}^{+\infty}
\alpha_\ell\underbrace{(2\ell+1)}_{\geq 2n-1}+c_{n}\sum_{\ell=n+1}^{+\infty}
\alpha_\ell\underbrace{(2\ell+1)}_{\geq 2n+1}\\
\nonumber
\geq&&\left(\alpha_0+c_{0}\sum_{\ell=n+1}^{+\infty}\alpha_\ell\right)
+\left(\alpha_1\cdot3+c_{1}\sum_{\ell=n+1}^{+\infty}\alpha_\ell\cdot3\right)\\
\nonumber
&&\dots +\left(\alpha_{n-1}\cdot(2n-1)+
c_{n-1}\sum_{\ell=n+1}^{+\infty}\alpha_\ell\cdot(2n-1)\right)\\
\nonumber
&&+\left(\alpha_{n}\cdot(2n+1)+c_{n}\sum_{\ell=n+1}^{+\infty}\alpha_\ell\cdot(2n+1)\right)\\
\label{disug43}
=&&\sum_{k=0}^n\underbrace{\left(\alpha_k+c_kR_n\right)}_{=1}(2k+1)=\sum_{k=0}^n(2k+1)=(n+1)^2.
\eeqs

\end{proof}

\begin{Rem}\label{rem3}
\begin{em}
As a consequence of Theorem~\ref{th2} we have that if $\{f_i\}_{i\in I}$ is such that $\| f_i\|=1$ for every $i\in I$, $\{g_j\}_{j\in J}$ is an orthonormal system in $L^2(\R)$ and
\beqsn
\langle\hat LW(f_i,g_j),W(f_i,g_j)\rangle\leq A,\qquad\forall i\in I,\ j\in J,
\eeqsn
for some constant $A>0$, then $J$ must be finite (while $I$ may be infinite).
\end{em}
\end{Rem}

\begin{Cor}
\label{cor4}
If $\{f_k\}_{k\in\N_0}$ is an orthonormal sequence in $L^2(\R)$, then
\beqs
\label{C41}
\sum_{k=0}^n\langle\hat LW(f_k),W(f_k)\rangle\geq(n+1)^2,\qquad\forall n\in\N_0,
\eeqs
and the estimate is optimal, in the sense that if $f_k$ are the Hermite functions then equality 
holds in
\eqref{C41} and, conversely, given $n_0\in\N$, if equality holds in \eqref{C41} for all $n\leq n_0$, then there exist
$c_k\in\C$ with $|c_k|=1$ such that $f_k=c_kh_k$ for all $0\leq k\leq n_0$.
\end{Cor}

\begin{proof}
The inequality \eqref{C41} is a particular case of Theorem~\ref{th2} for $g_k=f_k$.

In order to prove that the inequality is optimal we follow the same ideas as in
\cite[Thm.~2.3]{JP}. If
$f_k=h_k$ then \eqref{C41} is an equality by Lemma~\ref{lemma1}.

Now, if the equality holds in \eqref{C41} for all $0\leq n\leq n_0$, then for all 
$0\leq n\leq n_0$
\beqs
\nonumber
\langle\hat LW(f_n),W(f_n)\rangle=&&
\sum_{k=0}^n\langle\hat LW(f_k),W(f_k)\rangle
-\sum_{k=0}^{n-1}\langle\hat LW(f_k),W(f_k)\rangle\\
\label{C42}
=&&(n+1)^2-n^2=2n+1.
\eeqs

Since $\{\hat h_{j,k}\}_{j,k\in\N_0}=\{W(h_j,h_k)\}_{j,k\in\N_0}$ is an orthonormal basis in 
$L^2(\R^2)$ we have that
\beqsn
W(f_n)=\sum_{j,k=0}^{+\infty}\langle W(f_n),\hat h_{j,k}\rangle\hat h_{j,k},
\eeqsn
and hence, by \eqref{T22} and \eqref{Moyal}:
\beqs
\nonumber
&&\hskip-1cm\langle\hat L W(f_n),W(f_n)\rangle
=\sum_{j,k=0}^{+\infty}|\langle W(f_n),\hat h_{j,k}\rangle|^2(2k+1)\\
\nonumber
=&&\sum_{j,k=0}^{+\infty}|\langle W(f_n,f_n),W(h_j,h_k)\rangle|^2(2k+1)
=\sum_{j,k=0}^{+\infty}|\langle f_n,h_j\rangle|^2|\langle f_n,h_k\rangle|^2(2k+1)\\
\label{C43}
=&&\sum_{k=0}^{+\infty} \|f_n\|^2|\langle f_n,h_k\rangle|^2(2k+1)
=\sum_{k=0}^{+\infty} |\langle f_n,h_k\rangle|^2(2k+1).
\eeqs

We now proceed by induction on $n\in\N_0$.
From \eqref{C42} and \eqref{C43} for $n=0$ we have 
\beqsn
\sum_{k=0}^{+\infty}|\langle f_0,h_k\rangle|^2(2k+1)
=\langle\hat LW(f_0),W(f_0)\rangle=1
=\|f_0\|^2=\sum_{k=0}^{+\infty}|\langle f_0,h_k\rangle|^2,
\eeqsn
and hence
\beqsn
\langle f_0,h_k\rangle=0,\qquad\forall k\geq1,
\eeqsn
i.e. $f_0=c_0h_0$ for some $c_0\in\C$ with $|c_0|=1$, since
$\|f_0\|=\|h_0\|=1$.

Let us assume now that
\beqsn
f_k=c_kh_k,\quad c_k\in\C,\ |c_k|=1,\qquad k=0,1,\ldots,n-1,
\eeqsn
and let us prove that
\beqsn
f_n=c_nh_n,\quad c_n\in\C,\ |c_n|=1.
\eeqsn
Indeed,
\beqsn
\sum_{k=n}^{+\infty}|\langle f_n,h_k\rangle|^2(2k+1)=
\sum_{k=0}^{+\infty}|\langle f_n,h_k\rangle|^2(2k+1)
\eeqsn
since $\langle f_n,h_k\rangle=0$ for $0\leq k\leq n-1$ because $f_n$ is
orthogonal to $f_k=c_kh_k$ by inductive assumption.

Thus, by \eqref{C43} and \eqref{C42}, we have
\beqsn
\sum_{k=n}^{+\infty}|\langle f_n,h_k\rangle|^2(2k+1)=&&\langle\hat LW(f_n),W(f_n)\rangle
=2n+1=(2n+1)\|f_n\|^2\\
=&&(2n+1)\sum_{k=0}^{+\infty}|\langle f_n,h_k\rangle|^2
=\sum_{k=n}^{+\infty}(2n+1)|\langle f_n,h_k\rangle|^2
\eeqsn
again by inductive assumption.

Therefore $\langle f_n,h_k\rangle=0$ for all $k>n$ (and for $0\leq k\leq n-1$
by inductive assumption), which implies that $f_n=c_nh_n$ for some $c_n\in\C$ with
$|c_n|=1$ since $\|f_n\|=\|h_n\|=1$.
\end{proof}

From Corollary~\ref{cor4} we have, as in Remark~\ref{rem3}, that if
\beqsn
\langle\hat LW(f_j),W(f_j)\rangle\leq A,\qquad\forall j\in J,
\eeqsn
then $J$ must be finite.

Moreover, since
\beqs
\label{T51}
\langle\hat LW(f_k),W(f_k)\rangle=\mu^2(f_k)
+\mu^2(\hat f_k)+\Delta^2(f_k)+\Delta^2(\hat f_k)
\eeqs
by Lemma~\ref{lemma2}, we have obtained a simple proof of Theorem \ref{JP23} (the sharp
Mean-Dispersion Principle \cite[Thm. 2.3]{JP}), and then also of
Theorem~\ref{Shap-orig} (the original Shapiro's Mean-Dispersion Principle).

Formula \eqref{T51} says that Corollary \ref{cor4} is exactly a reformulation of Theorem \ref{JP23}, and in this sense Theorem \ref{th2} and Corollary \ref{cor4} can be seen as Mean-Dispersion principles related with the Wigner transform. On the other hand we observe that 
working with the Wigner transform gives several advantages. First of all we have more generality since in Theorem \ref{th2} we can consider different arguments $f_i, g_k$ in the cross-Wigner distribution; moreover the proofs with the Wigner transform are simpler and more self-contained with respect to \cite{JP}. Another advantage is that we have information on the Wigner transform of an orthonormal sequence $\{ f_k\}_{k\in\N_0}$ rather than on $f_k$ and $\hat{f}_k$ themselves, and this gives more possibilities on how such information can be treated and written. In Section \ref{sec5} we give a Mean-Dispersion principle on the trace of the covariance matrix associated to the Wigner transform; here we start by noting that, from Corollary~\ref{cor1}, the quantity
$
\mu^2(f_k)+\mu^2(\hat f_k)+\Delta^2(f_k)+\Delta^2(\hat f_k)
$
in \eqref{T51} can be written not only as $\langle\hat LW(f_k),W(f_k)\rangle$, but also through
many other operators, as we can see in the following examples.

\begin{Ex}
\label{ex1}
\begin{em}
For all $f\in L^2(\R)$ with $\|f\|=1$ and finite associated mean and variance of $f$ and $\hat f$
\beqsn
\mu^2(f)+\mu^2(\hat f)+\Delta^2(f)+\Delta^2(\hat f)=\langle M^2f,f\rangle+\langle D^2f,f\rangle
\eeqsn
by Corollary~\ref{cor1}$(a),(b)$. Therefore formula \eqref{m-v-intro} for an orthonormal sequence $\{ f_k\}_{k\in\N_0}$ in $L^2(\R)$ can be rewritten as
\beqsn
\sum_{k=0}^n\langle(M^2+D^2)f_k,f_k\rangle\geq(n+1)^2,\qquad
\forall n\in\N_0.
\eeqsn
\end{em}
\end{Ex}
\begin{Ex}
\label{ex2}
\begin{em}
For all $f\in L^2(\R)$ with $\|f\|=1$ and finite associated mean and variance of $f$ and $\hat f$
we have from Corollary~\ref{cor1}$(g),(h),(k),(l)$:
\beqsn
\langle[\frac14(D_1^2+D_2^2)+(M_1^2+M_2^2)]W(f),W(f)\rangle
=\Delta^2(f)+\Delta^2(\hat f)+\mu^2(f)+\mu^2(\hat f)
\eeqsn 
and hence for an orthonormal sequence $\{f_k\}_{k\in\N_0}\subset L^2(\R)$
\beqsn
\sum_{k=0}^n\langle PW(f_k),W(f_k)\rangle\geq(n+1)^2,\qquad\forall n\in\N_0,
\eeqsn
for $P=\frac14(D_1^2+D_2^2)+(M_1^2+M_2^2)$, by Theorem~\ref{JP23}.
\end{em}
\end{Ex}

We can also combine, for example, the operators of Examples~\ref{ex1} and \ref{ex2}, or
add combinations of $D_1$, $D_2$, $M_1D_1-M_2D_2$, by
Corollary~\ref{cor1}$(e),(f),(i),(j)$.

\section{Covariance}
\label{sec5}
In this section we give an uncertainty principle involving
the trace of the covariance matrix of the square of the Wigner distribution $|W(f)(x,\xi)|^2$, and explore its relations with Theorem~\ref{JP23}.

To this aim, let us first recall some notions about mean and covariance for a function of two variables $\rho(x,y)\in L^1(\R^2)$. We set
\beqs
\label{CO1}
\rho_X(x):=\int_\R\rho(x,y)dy,\qquad
\rho_Y(y):=\int_\R\rho(x,y)dx,
\eeqs
and then consider the {\em means}
\beqs
\label{CO2}
M(X):=\int_\R x\rho_X(x)dx,\qquad
M(Y):=\int_\R y\rho_Y(y)dy,
\eeqs
and the {\em covariances}
\beqsn
&&C(X,Y):=\int_{\R^2}(x-M(X))(y-M(Y))\rho(x,y)dxdy=C(Y,X)\\
&&C(X,X)=\int_{\R^2}(x-M(X))^2\rho(x,y)dxdy\\
&&C(Y,Y)=\int_{\R^2}(y-M(Y))^2\rho(x,y)dxdy.
\eeqsn
The {\em covariance matrix}
\beqsn
\left(\begin{matrix}
C(X,X)&C(X,Y)\\
C(Y,X)&C(Y,Y)
\end{matrix}\right)
\eeqsn
is symetric and its {\em trace} is given by
\beqs
\nonumber
C(X,X)+C(Y,Y)=&&\int_{\R^2}
\left((x-M(X))^2+(y-M(Y))^2\right)\rho(x,y)dxdy\\
\label{x2y2}
=&&\int_{\R^2}(x^2+y^2)\rho(x,y)dxdy\\
\nonumber
&&-2M(X)\int_{\R^2}x\rho(x,y)dxdy
\nonumber
-2M(Y)\int_{\R^2}y\rho(x,y)dxdy\\
\nonumber
&&+(M^2(X)+M^2(Y))\int_{\R^2}\rho(x,y)dxdy.
\eeqs
If $\rho(x,y)$ has null means $M(X)=M(Y)=0$, then \eqref{x2y2} represents the
trace of the covariance matrix of $\rho(x,y)$.

For $f\in L^2(\R)$ we can consider $\rho(x,\xi)=|W(f)(x,\xi)|^2\in L^1(\R^2)$ since
$W(f)\in L^2(\R^2)$.
It is then interesting to consider the quantity in \eqref{x2y2}
\beqsn
\int_{\R^2}(x^2+\xi^2)|W(f)(x,\xi)|^2dxd\xi,
\eeqsn
which is related to means and variances of $f$ and $\hat f$; indeed, if $f\in L^2(\R)$ with $\|f\|=1$, by Corollary~\ref{cor1}$(k),(l)$ we have
\beqs
\nonumber
&&\hskip-1cm\int_{\R^2}(x^2+\xi^2)|W(f)(x,\xi)|^2dxd\xi\\
\label{P54}
=&&\langle(M_1^2+M_2^2)W(f),W(f)\rangle\\
\label{P52}
=&&\mu^2(f)+\frac12\Delta^2(f)+\mu^2(\hat f)+\frac12\Delta^2(\hat f)\\
\label{P53}
\geq&&\frac12(\mu^2(f)+\mu^2(\hat f)+\Delta^2(f)+\Delta^2(\hat f))
\eeqs
and the equality in \eqref{P53} holds if and only if $\mu(f)=\mu(\hat f)=0$.
In particular, since the Hermite functions satisfy $\mu(h_k)=\mu(\hat h_k)=0$ by
\cite[Ex.~2.4]{JP}, from Theorem~\ref{JP23} we have the following:

\begin{Th}
\label{prop5}
 If $\{f_k\}_{k\in\N_0}$ is an orthonormal sequence in $L^2(\R)$, then
\beqs
\label{P51}
\sum_{k=0}^n\int_{\R^2}(x^2+\xi^2)|W(f_k)(x,\xi)|^2dxd\xi
\geq\frac{(n+1)^2}{2},\qquad\forall n\in\N_0.
\eeqs
Moreover, given $n_0\in\N$, the equality holds for all $n\leq n_0$ if and only if there exist $c_k\in\C$ with $|c_k|=1$
such that $f_k=c_kh_k$ for all $0\leq k\leq n_0$.
\end{Th}

\begin{proof}
The inequality \eqref{P51} immediately follows from
\eqref{P53} and Theorem~\ref{JP23}.
If $f_k$ are multiples of the Hermite functions $c_k h_k$ with $|c_k|=1$, then the equality holds because of \eqref{P52}, the fact that $\mu(h_k)=\mu(\hat h_k)=0$, and Theorem~\ref{JP23}.

In the other direction, if the equality holds in \eqref{P51} for all $n\leq n_0$, then from \eqref{P52} we have, for $n\leq n_0$,
\beqsn
\sum_{k=0}^n(\mu^2(f_k)+\mu^2(\hat f_k)+\frac12\Delta^2(f_k)+\frac12\Delta^2(\hat f_k))
=\frac{(n+1)^2}{2}
\eeqsn
and hence, from Theorem~\ref{JP23}:
\beqsn
\begin{cases}
\ds\mu(f_k)=\mu(\hat f_k)=0&\forall\,0\leq k\leq n\cr
\ds\sum_{k=0}^n(\Delta^2(f_k)+\Delta^2(\hat f_k))=(n+1)^2.
\end{cases}
\eeqsn

Then we conclude from Theorem~\ref{JP23}.
\end{proof}

Let us remark that from Theorem~\ref{prop5} we immediately get the following uncertainty principle for the covariance matrix:

\begin{Cor}
\label{cor5}
If $\{f_j\}_{j\in J}$ is an orthonormal sequence in $L^2(\R)$ with zero means
$\mu(f_j)=\mu(\hat f_j)=0$, and if the trace of the covariance matrix of $|W(f_j)(x,\xi)|^2$ is
uniformly bounded in $j$, we have
\beqsn
\int_{\R^2}(x^2+\xi^2)|W(f_j)(x,\xi)|^2dxd\xi\leq A,\qquad\forall j\in J,
\eeqsn
for some $A>0$. In particular,  $J$ is finite.
\end{Cor}

\begin{proof}
From Corollary~\ref{cor1}$(c),(d)$ we have that
\beqsn
M(X)=&&\int_{\R^2}x|W(f_k)(x,\xi)|^2dxd\xi
=\langle M_1W(f_k),W(f_k)\rangle=\mu(f_k)=0\\
M(Y)=&&\langle M_2W(f_k),W(f_k)\rangle=\mu(\hat f_k)=0
\eeqsn
by assumption, and hence from \eqref{x2y2}:
\beqsn
C(X,X)+C(Y,Y)=\int_{\R^2}(x^2+\xi^2)|W(f_k)(x,\xi)|^2dxd\xi.
\eeqsn
The thesis thus immediately follows from Theorem~\ref{prop5}.
\end{proof}

Note that Corollary~\ref{cor5} can be stated also in terms of the variances of 
$|W(f_j)(x,\xi)|^2$ since, in general, the {\em variances}
\beqsn
&&V(X)=\int_\R(x-M(X))^2\rho_X(x)dx,\\
&&V(Y)=\int_\R(y-M(Y))^2\rho_Y(y)dy,
\eeqsn
for $\rho_X,\rho_Y,M(X),M(Y)$ defined as in \eqref{CO1}-\eqref{CO2}, satisfy:
\beqsn
C(X,X)=V(X),\qquad C(Y,Y)=V(Y),
\eeqsn
if $\rho(x,y)\in L^1(\R^2)$.

\section{Cohen classes}
\label{sec4}

Infinitely many operators playing the same role as in the previous sections may be constructed by means of the {\em Cohen class}
\beqsn
Q(f,g)=\frac{1}{\sqrt{2\pi}}\,\sigma* W(f,g),\qquad f,g\in\Sch(\R),
\eeqsn
for some tempered distribution $\sigma\in\Sch'(\R^2)$. For $f,g\in\Sch(\R)$ we have $W(f,g)\in\Sch(\R^2)$, and then $Q(f,g)$ is well-defined for every $\sigma\in\Sch'(\R^2)$. As for the Wigner we define
\beqsn
Q[w]=\frac{1}{\sqrt{2\pi}}\,\sigma*\Wig[w],\qquad w\in\Sch(\R^2).
\eeqsn
If $\sigma=\F^{-1}(e^{-iP(\xi,\eta)})\in\Sch'(\R^2)$ for some polynomial 
$P\in\R[\xi,\eta]$ we have the following result (see \cite[Thms. 3.1 and 3.2]{BJO-Wigner}):

\begin{Th}
\label{th3132Wigner}
Let $B(x,y,D_x,D_y)$ be a linear partial differential operator with polynomial coefficients and let
$\sigma=\F^{-1}(e^{-iP(\xi,\eta)})\in\Sch'(\R^2)$ for some $P\in\R[\xi,\eta]$.
Then for every $w\in\Sch(\R^2)$:
\beqsn
(i)\hspace*{20mm}&&\hskip-1cmQ[B(M_1,M_2,D_1,D_2)w]\\
=&&B\left(M_1-\frac12D_2-P_1,M_1+\frac12 D_2-P_1,\frac12D_1+M_2-P_2,
\frac12D_1-M_2+P_2\right)Q[w]
\eeqsn
for
\beqs
\label{P1P2}
P_1=(iD_1P)(D_1,D_2),
\quad P_2=(iD_2P)(D_1,D_2).
\eeqs
\vspace*{1mm}
\beqsn
\hspace*{-15mm}(ii)\hspace*{30mm}
&&\hskip-1cmB(M_1,M_2,D_1,D_2)Q[w]\\
=&&Q\left[B\left(\frac{M_2+M_1}{2}+P_1^*,\frac{D_1-D_2}{2}+P_2^*,D_1+D_2,M_2-M_1\right)
w\right]
\eeqsn
for
\beqsn
P_1^*=(iD_1P)(D_1+D_2,M_2-M_1), \quad
P_2^*=(iD_2P)(D_1+D_2,M_2-M_1).
\eeqsn
\end{Th}

Let us remark that if $\sigma=\F^{-1}(e^{-iP(\xi,\eta)})$ then $|\hat\sigma|=1$ and hence,
for all $f_1,f_2,g_1,g_2\in\Sch(\R)$, from \eqref{Parceval} and \eqref{Moyal}:
\beqs
\nonumber
\langle Q(f_1,g_1),Q(f_2,g_2)\rangle=&&\frac{1}{{2\pi}}\langle\sigma *
W(f_1,g_1),\sigma*W(f_2,g_2)\rangle\\
\nonumber
=&&\frac{1}{{2\pi}}\langle\F^{-1}(\sqrt{2\pi}\,\hat\sigma\cdot\widehat{W(f_1,g_1)}),
\F^{-1}(\sqrt{2\pi}\,\hat\sigma\cdot\widehat{W(f_2,g_2)})\rangle\\
\nonumber
=&&\langle\hat\sigma\cdot\widehat{W(f_1,g_1)},\hat\sigma\cdot\widehat{W(f_2,g_2)}\rangle\\
\nonumber
=&&\langle|\hat\sigma|^2\widehat{W(f_1,g_1)},\widehat{W(f_2,g_2)}\rangle
=\langle\widehat{W(f_1,g_1)},\widehat{W(f_2,g_2)}\rangle\\
\label{QW}
=&&\langle W(f_1,g_1),W(f_2,g_2)\rangle
=\langle f_1,f_2\rangle\overline{\langle g_1,g_2\rangle},
\eeqs
since $\widehat{f*g}=\sqrt{2\pi}\hat f\cdot\hat g$.

Moreover:
\begin{Th}
\label{th9}
Let $\{f_k\}_{k\in\N_0}\subset\Sch(\R)$ be an orthonormal basis in $L^2(\R)$. Then
$\{Q(f_j,f_k)\}_{j,k\in\N_0}$ is an orthonormal basis in $L^2(\R^2)$.
\end{Th}

\begin{proof}
Let us first remark that $Q(f_j,f_k)\in\Sch(\R^2)\subset L^2(\R^2)$.
Moreover $\{Q(f_j,f_k)\}_{j,k\in\N_0}$ is an orthonormal sequence by \eqref{QW}.

We only have to prove that if $F\in L^2(\R^2)$ satisfies
\beqsn
\langle F,Q(f_j,f_k)\rangle=0,\qquad\forall j,k\in\N_0,
\eeqsn
then $F=0$ a.e. in $\R^2$ (see \cite[Thm. 3.4.2]{C}).
Let $G=\F^{-1}(\hat F/\hat\sigma)\in L^2(\R^2)$, so that from \eqref{Parceval}
\beqsn
0=&&\langle F,Q(f_j,f_k)\rangle=
\langle \hat F,\widehat{Q(f_j,f_k)}\rangle
=\langle \hat G\cdot\hat\sigma,\hat\sigma\cdot\widehat{W(f_j,f_k)}\rangle\\
=&&\langle |\hat\sigma|^2\hat G,\widehat{W(f_j,f_k)}\rangle
=\langle\hat G,\widehat{W(f_j,f_k)}\rangle
=\langle G,{W(f_j,f_k)}\rangle,\qquad\forall j,k\in\N_0,
\eeqsn
which implies $G=0$ a.e. in $\R^2$ since 
$\{W(f_j,f_k)\}_{j,k\in\N_0}$ is a basis in $L^2(\R^2)$ by Theorem~\ref{th1}.

Then $\hat F=\hat G\cdot\hat\sigma=0$, i.e. $F=0$ a.e. in $\R^2$.
\end{proof}

Let us remark that if $f,g\in\Sch(\R)\subset L^2(\R)$ with $\|f\|=\|g\|=1$ then,
by Lemma~\ref{lemma2}, \eqref{QW} and Theorem~\ref{th3132Wigner}, we have 
\beqs
\nonumber
&&\hskip-1cm\mu^2(g)+\mu^2(\hat g)+\Delta^2(g)+\Delta^2(\hat g)
=\langle\hat L W(f,g),W(f,g)\rangle\\
\nonumber
=&&\langle W(f,(M^2+D^2)g),W(f,g)\rangle
=\langle Q(f,(M^2+D^2)g),Q(f,g)\rangle\\
\label{Q1}
=&&\langle[(M_1+\frac12D_2-P_1)^2+(\frac12D_1-M_2+P_2)^2]Q(f,g),Q(f,g)\rangle
\eeqs
for $P_1,P_2$ as in \eqref{P1P2}.

Then Theorem \ref{th2} can be rephrased as follows, for any choice of $P\in\R[\xi,\eta]$:

\begin{Th}
\label{th6}
Let $\{f_k\}_{k\in\N_0},\{g_k\}_{k\in\N_0}\subset\Sch(\R)$ be two orthonormal sequences in 
$L^2(\R)$. Then
\beqs
\label{T61}
\sum_{k=0}^n\langle\tilde LQ(f_j,g_k),Q(f_j,g_k)\rangle\geq(n+1)^2,
\qquad\forall n\in\N_0,
\eeqs
for any linear partial differential operator $\tilde L$ of the form
\beqsn
\tilde L(M_1,M_2,D_1,D_2)=\left(M_1+\frac12D_2-P_1\right)^2+\left(\frac12D_1-M_2+P_2\right)^2
\eeqsn
with
\beqsn
&&P_1=(iD_1P)(D_1,D_2),\qquad P_2=(iD_2P)(D_1,D_2),\\
&&P\in\R[\xi,\eta],\qquad \sigma=\F^{-1}(e^{-iP(\xi,\eta)}),\\
&&Q(f_j,f_k)=\frac{1}{\sqrt{2\pi}}\,\sigma* W(f_j,f_k).
\eeqsn
\end{Th}

\begin{Ex}
\label{ex7}
\begin{em}
Let $P(D_1,D_2)=\frac12 D_1D_2$. Then
\beqsn
&&P_1=iD_1P(\xi_1,\xi_2)\big|_{(\xi_1,\xi_2)=(D_1,D_2)}=\frac12D_2\\
&&P_2=iD_2P(\xi_1,\xi_2)\big|_{(\xi_1,\xi_2)=(D_1,D_2)}=\frac12D_1
\eeqsn
and hence
\beqsn
\tilde L=&&\left(M_1+\frac12D_2-\frac12D_2\right)^2
+\left(\frac12D_1-M_2+\frac12D_1\right)^2\\
=&&M_1^2+(D_1-M_2)^2.
\eeqsn
Therefore, by Theorem~\ref{th6}, we obtain
\beqsn
\sum_{k=1}^n\langle(M_1^2+(D_1-M_2)^2)Q(f_j,f_k),Q(f_jf_k)\rangle
\geq (n+1)^2,\qquad
\forall n\in\N_0.
\eeqsn
\end{em}
\end{Ex}

\begin{Ex}
\begin{em}
Similar results can be obtained considering the operator
$P(M_1,M_2)=M_1^2+M_2^2$ in \eqref{P54} instead of $\hat L$ and then
Theorem~\ref{prop5} instead of Corollary~\ref{cor4}.
Indeed, for $f\in \Sch(\R)$ with $\|f\|=1$ we can write, by
Proposition~\ref{prop1}, \eqref{QW} and Theorem~\ref{th3132Wigner}:
\beqsn
&&\hskip-1cm\langle(M_1^2+M_2^2)W(f),W(f)\rangle\\
=&&\langle\Wig[(\frac14(M_1+M_2)^2+\frac14(D_1-D_2)^2)f\otimes\bar f],W(f)\rangle\\
=&&\frac14\langle Q[((M_1+M_2)^2+(D_1-D_2)^2)f\otimes\bar f],Q(f)\rangle\\
=&&\frac14\langle(M_1-\frac12D_2-P_1+M_1+\frac12D_2-P_1)^2Q(f),Q(f)\rangle\\
&&+\frac14\langle
(\frac12D_1+M_2-P_2-\frac12D_1+M_2-P_2)^2Q(f),Q(f)\rangle\\
=&&\langle((M_1-P_1)^2+(M_2-P_2)^2)Q(f),Q(f)\rangle
\eeqsn
for any $P_1,P_2$ as in \eqref{P1P2}.

It follows that if $\{f_k\}_{k\in\N_0}\subset\Sch(\R)$ is an orthonormal sequence in $L^2(\R)$ then, from
Theorem~\ref{prop5},
\beqs
\label{L*}
\sum_{k=0}^n\langle L^*Q(f_k),Q(f_k)\rangle
\geq\frac{(n+1)^2}{2},
\qquad\forall n\in\N_0,
\eeqs
for any linear partial differential operator $L^*$ of the form
\beqsn
L^*(M_1,M_2,D_1,D_2)=(M_1-P_1)^2+(M_2-P_2)^2
\eeqsn
with
\beqsn
&&P_1=(iD_1P)(D_1,D_2),\qquad P_2=(iD_2P)(D_1,D_2),\\
&&P\in\R[\xi,\eta],\qquad \sigma=\F^{-1}(e^{-iP(\xi,\eta)}),\\
&&Q(f_k)=\frac{1}{\sqrt{2\pi}}\,\sigma* W(f_k).
\eeqsn
\end{em}
\end{Ex}

\begin{Rem}
\label{rem66}
\begin{em}
Any linear operator $T:\ L^2(\R^2)\to L^2(\R^2)$ (not necessarily everywhere defined) satisfying, for some orthonormal sequence
$\{f_k\}_{k\in\N_0}\subset L^2(\R)$,
\beqs
\label{pdo-notbound}
\sum_{k=0}^n\langle TW(f_j,f_k),W(f_j,f_k)\rangle\geq(n+1)^2,\qquad
\forall n\in\N_0,
\eeqs
cannot be a bounded operator on $L^2(\R^2)$.
Indeed, assuming by contradiction that $T$ is bounded, by Theorem~\ref{th1} we would have, for all $n\in\N_0$:
\beqsn
(n+1)^2\leq&&\sum_{k=0}^n\langle TW(f_j,f_k),W(f_j,f_k)\rangle\\
\leq&&\sum_{k=0}^n\|T\|_{\Lin(L^2,L^2)}\|W(f_j,f_k)\|_{L^2}^2
=(n+1)\|T\|_{\Lin(L^2,L^2)}
\eeqsn
which gives a contradiction for large $n$. 
The above considerations can be applied to the partial differential operators with polynomial coefficients appearing in the various results were we have proved estimates of the kind of \eqref{pdo-notbound}. This is not surprising since all
non-constant differential operators with polynomial coefficients are in fact unbounded in $L^2(\R^n)$.
Indeed, assume first that $P(x,D)$ has non-constant coefficients, i.e.
\beqsn
P(x,D)=\sum_{|\beta|\leq\ell}P_\beta(x)D^\beta,\qquad x\in\R^n,
\eeqsn
with $P_\beta(x)$ polynomials of degree less than or equal to $m\geq1$.
We choose $\beta_0\in\N_0^n$, $|\beta_0|\leq\ell$ and $a\in\R^n\setminus\{0\}$
such that $P_{\beta_0}(ta)$ is a polynomial in $t$ of maximum degree $m$. 

Taking then $\varphi\in\D(\R^n)$ with
\beqsn
D^\beta\varphi(0)=
\begin{cases}
0, & |\beta|\leq\ell,\beta\neq\beta_0\cr
1,&\beta=\beta_0,
\end{cases}
\eeqsn
we have 
$\|\varphi(x-ta)\|_{L^2}=\|\varphi(x)\|_{L^2},$ but 
\beqsn
\|P(x,D)\varphi(x-ta)\|_{L^2}=\|P(x+ta,D)\varphi(x)\|_{L^2}\to+\infty\quad\mbox{as}\ t\to+\infty.
\eeqsn

If $P(x,D)=P(D)$ has constant coefficients we argue similarly, choosing $a\in\R^n\setminus\{0\}$ 
in such a way that $P(ta)$ is a polynomial in $t$ of maximum degree $m\geq1$ and taking then
 $\varphi\in\D(\R^n)$ with $\hat\varphi(0)\neq0$ we have 
$\|e^{itx\cdot a}\varphi(x)\|_{L^2}=\|\varphi(x)\|_{L^2},$
but
\beqsn
\|P(D)e^{itx\cdot a}\varphi(x)\|_{L^2}=\|P(\xi)\hat{\varphi}(\xi-ta)\|_{L^2}
=\|P(\xi+ta)\hat{\varphi}(\xi)\|_{L^2}\to+\infty \quad\mbox{as}\ t\to+\infty.
\eeqsn
\end{em}
\end{Rem}

\section{Riesz bases}
\label{sec7}

In this section we consider a general Riesz basis of $L^2(\R)$ instead of an orthonormal basis.
We recall that a {\em Riesz basis} in a Hilbert space $H$ is the image of an orthonormal basis
for $H$ under an invertible linear bounded operator.
In particular, if $\{u_k\}_{k\in\N_0}$ is a Riesz basis for $L^2(\R)$, we can find an invertible linear bounded operator $U_1:\,L^2(\R)\to L^2(\R)$ such that
\beqsn
U_1(u_k)=h_k,\qquad\forall k\in\N_0,
\eeqsn
for the Hermite functions $\{h_k\}_{k\in\N_0}$. Moreover (see \cite[Lemma 3.6.2]{C})
\beqs
\label{c>0}
0<C_1:=\inf_{k\in\N_0}\|u_k\|^2\leq \sup_{k\in\N_0}\|u_k\|^2=:C_2<+\infty.
\eeqs
We can thus generalize Theorem \ref{th2} to Riesz bases:

\begin{Th}
\label{th10}
If $\{u_k\}_{k\in\N_0}$ and $\{v_k\}_{k\in\N_0}$ are Riesz bases for $L^2(\R)$ and $\hat L$ is the operator in \eqref{Lhat}, then for all $i,n\in\N_0$
\beqs
\label{Rb-est}
\sum_{k=0}^n\langle\hat LW(u_i,v_k),W(u_i,v_k)\rangle\geq
\frac{\|U_2^{-1}\|_{\mathcal L(L^2,L^2)}^2}{\|U_1\|_{\mathcal L(L^2,L^2)}^2}
\left[\frac{n+1}{\|U_2^{-1}\|_{\mathcal L(L^2,L^2)}^2 \|U_2\|_{\mathcal L(L^2,L^2)}^2}
\right]^2,
\eeqs
where $U_j:\,L^2(\R)\to L^2(\R)$, $j=1,2$, are such that $U_1(u_k)=h_k$ and $U_2(v_k)=h_k$, for
the Hermite functions $h_k$ defined in \eqref{hermite}, and $[x]$ denotes the integer part of $x$.
\end{Th}

\begin{proof}
As in \eqref{T22}-\eqref{T23} we obtain that
\beqsn
\sum_{k=0}^n \langle\hat{L}W(u_i,v_k),W(u_i,v_k)\rangle =&& \sum_{j=0}^{+\infty} |\langle u_i,h_j\rangle|^2 \sum_{\ell=0}^{+\infty} \sum_{k=0}^n |\langle v_k,h_\ell\rangle|^2 (2\ell +1) \\
=&& \| u_i\|^2 \sum_{\ell=0}^{+\infty} \sum_{k=0}^n |\langle v_k,h_\ell\rangle|^2 (2\ell +1),
\eeqsn
and we can suppose that the series in the right-hand side is convergent, otherwise \eqref{Rb-est} would be trivial.
We thus obtain, for the constant $C_1$ defined in \eqref{c>0}:
\beqs
\label{71}
\sum_{k=0}^n\langle\hat LW(u_i,v_k),W(u_i,v_k)\rangle
\geq C_1\sum_{\ell=0}^{+\infty}
\sum_{k=0}^n|\langle v_k,h_\ell\rangle|^2(2\ell+1)=C_1\sum_{\ell=0}^{+\infty}\alpha_\ell
(2\ell+1)
\eeqs
for
\beqsn
\alpha_\ell:=\sum_{k=0}^n|\langle v_k,h_\ell\rangle|^2\leq
\sum_{k=0}^{+\infty}|\langle v_k,h_\ell\rangle|^2\leq B \|h_\ell\|^2=B
\eeqsn
for $B=\|U_2^{-1}\|^2_{\mathcal L(L^2,L^2)}$ because of \cite[Prop. 3.6.4]{C}.

We have
\beqsn
\sum_{\ell=0}^{+\infty}\alpha_\ell=&&\sum_{\ell=0}^{+\infty}\sum_{k=0}^n|\langle v_k,h_\ell\rangle|^2
=\sum_{k=0}^n\sum_{\ell=0}^{+\infty}|\langle v_k,h_\ell\rangle|^2
=\sum_{k=0}^n\|v_k\|^2\geq\tilde C_1(n+1),
\eeqsn
for $\tilde{C}_1:=\inf_{k\in\N_0} \| v_k\|^2$.
Note that $0<\tilde C_1\leq\sup_{k\in\N_0}\|U_2^{-1}\|^2_{\mathcal L(L^2,L^2)}\|h_k\|^2=B$.

Let us now assume $n\geq\frac{B}{\tilde C_1}-1$ so that 
\beqsn
m:=\left[\frac{n+1}{B}\tilde C_1\right]-1\in\N_0
\eeqsn
and write
\beqsn
\tilde C_1(n+1)\leq\sum_{\ell=0}^{+\infty}\alpha_\ell=\alpha_0+\ldots+\alpha_m+R_m
\eeqsn
with
\beqsn
R_m:=\sum_{\ell\geq m+1}\alpha_\ell.
\eeqsn

If $R_m=0$ then $\frac{n+1}{B}\tilde C_1\in\N$ and $\alpha_0=\ldots=\alpha_m=B$
because otherwise or $m+1<\frac{n+1}{B}\tilde C_1$ or $\alpha_k<B$ for
some $k=0,\ldots,m$ and 
\beqsn
\tilde C_1(n+1)\leq\alpha_0+\ldots+\alpha_m<B\cdot \frac{n+1}{B}\tilde C_1=(n+1)\tilde C_1
\eeqsn
would give a contradiction.
It follows that setting
\beqsn
c_k:=
\begin{cases}
0 & \mbox{if}\ R_m=0\cr
\frac{B-\alpha_k}{R_m}& \mbox{if}\ R_m>0,
\end{cases}
\eeqsn
we have $c_k\geq0$ and $\alpha_k+c_kR_m=B$ for all $0\leq k\leq m$.

Moreover, if $R_m>0$
\beqsn
(c_0+\ldots+c_m)R_m=&&B\left[\frac{n+1}{B}\tilde C_1\right]-(\alpha_0+\ldots+\alpha_m)\\
\leq&& \tilde C_1(n+1)-(\alpha_0+\ldots+\alpha_m)\leq R_m.
\eeqsn

It follows that $c_0+\ldots+c_m\leq1$ and hence, for all $n\geq\frac{B}{\tilde C_1}-1$,
\beqsn
\sum_{\ell=0}^{+\infty}\alpha_\ell(2\ell+1)=&&\sum_{\ell=0}^{m}\alpha_\ell (2\ell+1)
+\sum_{\ell\geq m+1}\alpha_\ell (2\ell+1)\\
\geq&&\sum_{\ell=0}^{m}\alpha_\ell (2\ell+1)
+(c_0+\ldots+c_m) \sum_{\ell\geq m+1}\alpha_\ell (2\ell+1)\\
=&&\sum_{\ell=0}^{m}\alpha_\ell (2\ell+1)+ c_{0}
\sum_{\ell\geq m+1}
\alpha_\ell
\underbrace{(2\ell+1)}_{\geq1}+\ldots
+ c_{m}\sum_{\ell\geq m+1}\alpha_\ell
\underbrace{(2\ell+1)}_{\geq 2m+1}\\
\geq&&\underbrace{(\alpha_0
+ c_{0}R_m)}_{=B}\cdot1+\ldots
+\underbrace{(\alpha_{m}
+ c_{m}R_m)}_{=B}\cdot (2m+1)\\
=&&B\sum_{k=0}^{m}(2k+1)=B(m+1)^2=B\left[\frac{n+1}{B}\tilde C_1\right]^2.
\eeqsn
Note that the above inequality is trivial if $\frac{n+1}{B}\tilde C_1<1$ so that
 from \eqref{71} we have, for all $n\in \N_0$,
\beqs
\label{CC1}
\sum_{k=0}^n\langle\hat LW(u_i,v_k),W(u_i,v_k)\rangle
\geq 
C_1B\left[\frac{n+1}{B}\tilde C_1\right]^2.
\eeqs

Let us now remark that, from the the continuity of $U_j:\,L^2(\R)\to L^2(\R)$, $j=1,2$,
\beqsn
1=\|h_k\|_{L^2}\leq\|U_1\|_{\mathcal L(L^2,L^2)}\cdot\|u_k\|_{L^2}, \quad 1=\|h_k\|_{L^2}\leq\|U_2\|_{\mathcal L(L^2,L^2)}\cdot\|v_k\|_{L^2}
\eeqsn
for every $k\in\N_0$, and therefore
\beqsn
C_1=\inf_{k\in\N_0}\|u_k\|_{L^2}^2\geq
\frac{1}{\|U_1\|_{\mathcal L(L^2,L^2)}^2},\quad
\tilde{C}_1=\inf_{k\in\N_0}\|v_k\|_{L^2}^2\geq
\frac{1}{\|U_2\|_{\mathcal L(L^2,L^2)}^2}.
\eeqsn
Since $B=\|U_2^{-1}\|_{\mathcal L(L^2,L^2)}^2$ we finally have from \eqref{CC1} that
\beqsn
\sum_{k=0}^n\langle\hat LW(u_i,v_k),W(u_i,v_k)\rangle
\geq\frac{\|U_2^{-1}\|_{\mathcal L(L^2,L^2)}^2}{\|U_1\|_{\mathcal L(L^2,L^2)}^2}
\left[\frac{n+1}{\|U_2^{-1}\|_{\mathcal L(L^2,L^2)}^2\|U_2\|_{\mathcal L(L^2,L^2)}^2}\right]^2
\eeqsn
for all $n\in\N_0$.
\end{proof}

From Theorem~\ref{th10} and Lemma~\ref{lemma2} we have the mean-dispersion
principle for Riesz bases:
\begin{Cor}
\label{cor11}
Let $\{u_k\}_{k\in\N_0}$ be a Riesz basis in $L^2(\R)$, with
$U(u_k)=h_k$, for the Hermite functions $\{h_k\}_{k\in\N_0}$ defined in 
\eqref{hermite}.
Then for all $n\in\N_0$
\beqsn
\sum_{k=0}^n(\Delta^2(u_k)+\Delta^2(\hat u_k)+\mu^2(u_k)+\mu^2(\hat u_k))
\geq\frac{1}{\|U^{-1}\|^2\|U\|^2}
\left[\frac{n+1}{\|U^{-1}\|^2\|U\|^2}\right]^2,
\eeqsn
where $\|\cdot\|=\|\cdot\|_{\mathcal L(L^2,L^2)}$.
\end{Cor}

\begin{proof}
From Lemma~\ref{lemma2} we have that
\beqsn
\langle\hat LW(u_k),W(u_k)\rangle=\|u_k\|^4(\Delta^2(u_k)+\Delta^2(\hat u_k)+\mu^2(u_k)+\mu^2(\hat u_k)).
\eeqsn
Since $\| u_k\|\leq \|U^{-1}\|\cdot\| h_k\|=\| U^{-1}\|$, by Theorem \ref{th10} we obtain:
\beqsn
&&\hskip-1cm\sum_{k=0}^n(\Delta^2(u_k)+\Delta^2(\hat u_k)+\mu^2(u_k)+\mu^2(\hat u_k))\\
\geq&&\frac{1}{\| U^{-1}\|^4}\sum_{k=0}^n\|u_k\|^4(\Delta^2(u_k)+\Delta^2(\hat u_k)+\mu^2(u_k)+\mu^2(\hat u_k))\\
\geq&&\frac{1}{\|U^{-1}\|^2\|U\|^2}
\left[\frac{n+1}{\|U^{-1}\|^2\|U\|^2}\right]^2.
\eeqsn

\end{proof}

Note that if the Riesz basis $\{u_k\}_{k\in\N_0}$ is orthonormal then $\|U\|=1$ since
\beqsn
U(f)=\sum_{k=0}^{+\infty}\langle f,u_k\rangle U(u_k)
=\sum_{k=0}^{+\infty}\langle f,u_k\rangle h_k,\qquad f\in L^2(\R),
\eeqsn
and hence $\|U(f)\|_{L^2}=\|f\|_{L^2}$ for all $f\in L^2(\R)$.

From Corollary~\ref{cor11} in the case of orthonormal Riesz bases we thus find again \eqref{m-v-intro}, i.e. Shapiro's Mean Dispersion principle.
This improves \cite[Cor. 2.8]{JP} for $\|U\|=1$, where a weaker estimate is obtained with respect to Shapiro's Mean Dispersion principle (\cite[Thm. 2.3]{JP}).

\vspace*{10mm}
{\bf Acknowledgments.}
The authors are sincerely grateful to Prof. Paolo Boggiatto for helpful discussions about the subject of this work. 
Boiti and Oliaro were partially supported by  the INdAM-GNAMPA Project 2020 ``Analisi microlocale e applicazioni: PDEs stocastiche e di evoluzione, analisi tempo-frequenza, variet\`a", and by the Project FAR 2019 (University of Ferrara). Boiti was also supported by the Projects FAR 2020, FAR 2021, FIRD 2022 (University of Ferrara) and Oliaro was also supported by the Projects ``Ricerca locale 2020 - linea A'' and ``Ricerca locale 2021 - linea A'' (University of Torino). Jornet was partially supported by the projects
	PID2020-119457GB-100 funded by MCIN/AEI/10.13039/501100011033 and
	by ``ERDF A way of making Europe'', and by the project GV AICO/2021/170.

\end{document}